\documentclass[12pt]{amsart}

\usepackage{amsmath,amsthm,amssymb,mathdots,fullpage,enumerate,color}

\usepackage{xcolor}
\usepackage[backref=page]{hyperref}
\definecolor{color1}{rgb}{0.05,.4,0.05}
\definecolor{color2}{rgb}{0.05,0.05,0.4}
\hypersetup{
bookmarksopen,bookmarksnumbered,
colorlinks,
linkcolor=color2,citecolor=color1,
}
\newtheorem{theorem}{Theorem}[section]
\newtheorem{lemma}[theorem]{Lemma}
\newtheorem{prop}[theorem]{Proposition}

\newtheorem{cor}[theorem]{Corollary}

\newtheorem{conj}{Conjecture}

\theoremstyle{remark}
\newtheorem*{remark}{Remark}

\newcommand{\R}{\mathbb R}
\newcommand{\C}{\mathbb C}

\newcommand{\Z}{\mathbb Z}
\newcommand{\Q}{\mathbb Q}

\newcommand{\A}{\mathbb A}

\newcommand{\p}{\mathfrak p}
\newcommand{\g}{\mathfrak g}
\newcommand{\e}{\mathfrak e}

\newcommand{\gk}{\mathfrak k}
\newcommand{\gq}{\mathfrak q}
\newcommand{\ga}{\mathfrak a}

\newcommand{\gn}{\mathfrak n}
\newcommand{\gt}{\mathfrak t}
\newcommand{\gl}{\mathfrak l}
\newcommand{\gu}{\mathfrak u}
\newcommand{\gm}{\mathfrak m}

\newcommand{\cO}{\mathcal O}

\newcommand{\cS}{\mathcal S}

\newcommand{\cH}{\mathcal H}

\newcommand{\GL}{\text{GL} }

\newcommand{\ad}{\text{ad}}
\newcommand{\Ad}{\text{Ad}}

\newcommand{\tr}{\text{tr}}

\newcommand{\be}{\begin{equation}}
\newcommand{\ee}{\end{equation}}
\newcommand{\bes}{\begin{equation*}}
\newcommand{\ees}{\end{equation*}}

\newcommand{\ba}{\begin{eqnarray}}
\newcommand{\ea}{\end{eqnarray}}
\newcommand{\bas}{\begin{eqnarray*}}
\newcommand{\eas}{\end{eqnarray*}}

\title{Endoscopy and cohomology of $U(n,1)$}

\author{Simon Marshall}
\address{Department of Mathematics\\
University of Wisconsin -- Madison\\
480 Lincoln Drive\\
Madison\\
WI 53706, USA}
\email{marshall@math.wisc.edu}
\thanks{Supported by NSF grant DMS-1501230.}

\author{Sug Woo Shin}
\address{Department of Mathematics\\
University of California, Berkeley\\
901 Evans Hall, Berkeley, CA 94720, USA
/ Korea Institute for Advanced Study,\\
 Dongdaemun-gu, Seoul 130-722, Republic of Korea}
\email{sug.woo.shin@berkeley.edu}
\thanks{Supported by NSF grant DMS-1501882 and a Sloan Fellowship.}

\begin{document}

\begin{abstract}

By assuming the endoscopic classification of automorphic representations on inner forms of unitary groups, which is currently work in progress by Kaletha, Minguez, Shin, and White, we bound the growth of cohomology in congruence towers of locally symmetric spaces associated to $U(n,1)$.  In the case of lattices arising from Hermitian forms, we conjecture that the growth exponents we obtain are sharp in all degrees.

\end{abstract}

\date{\today}

\maketitle

\section{Introduction}

This paper studies the limit multiplicity problem for cohomological automorphic forms on arithmetic quotients of $U(N-1,1)$.  Let $F$ be a totally real number field with ring of integers $\cO_F$. Write $\A$ for the ring of adeles over $F$.  Let $E$ be a totally imaginary quadratic extension of $F$.  Let $G$ be a unitary group with respect to $E/F$, that is, an inner form of the quasi-split unitary group $U(N)=U_{E/F}(N)$ with signature $(N-1,1)$ at one real place and compact factors at all other real places.  Let $S$ be a finite set of places to be defined later, and which includes all primes above 2, 3, and $\infty$, and let $\gn \subset \cO_F$ be a nonzero ideal that is divisible only by primes away from $S$ that split in $E/F$.  We let $K(\gn) \subset G(\A_f)$ be the compact congruence subgroup of level $\gn$, and let $\Gamma(\gn) = G(F) \cap K(\gn)$ be the congruence arithmetic lattice in $U(N-1,1)$ of level $\gn$ associated to $G$.  Let $Y(\gn)$ be the manifold $\Gamma(\gn) \backslash U(N-1,1) / U(N-1) \times U(1)$, which is a connected finite volume complex hyperbolic manifold of complex dimension $N-1$. (See \eqref{X(n)} below for the precise definition.) Write $h^d_{(2)}(Y(\gn))$ for the dimension of the $L^2$-cohomology of $Y(\gn)$ in degree $d\ge 0$.

\begin{theorem}\label{simplemaintheorem}

Assume the endoscopic classification for inner forms of $U(N)$ stated in Theorem 1.7.1 of \cite{KMSW}.  If $d < N-1$, we have $h^d_{(2)}(Y(\gn)) \ll_\epsilon \textup{vol}(Y(\gn))^{Nd / (N^2-1) + \epsilon}$.

\end{theorem}

The case $d>N-1$ follows by Poincar\'e duality. It is well known that $h^{N-1}_{(2)}(Y(\gn)) \sim \text{vol}(Y(\gn))$.  Previous results of this type in the case of $U(2,1)$ and $U(2,2)$ can be found in work of the first author \cite{Ma1,Ma2}.

Theorem \ref{simplemaintheorem} fits into the general framework of estimating the asymptotic multiplicities of automorphic forms.  We now recall the general formulation of this problem, and some of the previous results on it.  Let $G$ be a semisimple real algebraic group with no compact factors.  If $\Gamma \subset G$ is a lattice and $\pi \in \widehat{G}$, we let $m(\pi, \Gamma)$ be the multiplicity with which $\pi$ appears in $L^2(\Gamma \backslash G)$.  If we now assume that $\Gamma$ is congruence arithmetic and that $\Gamma_n \subset \Gamma$ is a family of principal congruence subgroups, the limit multiplicity problem is to provide estimates for $m(\pi, \Gamma_n)$.

A general principle that has emerged from work on this problem is that, the further $\pi$ is from being discrete series, the better bounds one should be able to prove for $m(\pi, \Gamma_n)$.  If we define $V(n) = \text{vol}(\Gamma_n \backslash G)$, the trivial bound (at least when $\Gamma$ is cocompact) is $m(\pi, \Gamma_n) \ll V(n)$, and it is known from work of de George and Wallach \cite{dGW} (if $\Gamma$ is cocompact) and Savin \cite{Sa} (if it is not) that this is realized if and only if $\pi$ is in the discrete series.  In the cocompact case, it also follows from \cite{dGW} that if $\pi$ is nontempered, then one has a bound of the form $m(\pi, \Gamma_n) \ll V(n)^{1-\delta(\pi)}$ for some $\delta(\pi) > 0$; see the introduction of \cite{SX} for an explanation of this principle, and \cite{X} for an explicit determination of such a $\delta(\pi)$ in some cases.

For the most highly nontempered representation, namely the trivial one, one has $m(\pi, \Gamma_n) = 1$.  Sarnak and Xue \cite{SX} made a conjecture that interpolates between this and $m(\pi,\Gamma_n) \ll V(n)$ in the discrete series case.  Define $p(\pi)$ to be the infimum over $p$ for which the $K$-finite matrix coefficients of $\pi$ lie in $L^p(G)$.  We then have:

\begin{conj}[Sarnak-Xue]
\label{SXconj}

For fixed $\pi$, we have $m(\pi, \Gamma_n) \ll_\epsilon V(n)^{2/p(\pi) + \epsilon}$.

\end{conj}

Note that Conjecture \ref{SXconj} is weaker than the trivial bound in both cases of $\pi$ discrete or trivial.  The point is that it is much stronger for general nontempered $\pi$ than what one can prove using the methods of deGeorge--Wallach mentioned above.  Sarnak and Xue established Conjecture \ref{SXconj} for $SL(2,\R)$ or $SL(2,\C)$, and proved an approximation for $SU(2,1)$ that, in our setting, implies that $h^1(Y(\gn)) \ll \text{vol}(Y(\gn))^{7/12 + \epsilon}$ when $N = 3$ and $\Gamma$ is cocompact and arises from a Hermitian form.

We show in Proposition \ref{cohreps} that the representations $\pi$ of $U(N-1,1)$ contributing to $h^d_{(2)}(Y(\gn))$ all have $p(\pi) \ge 2(N-1)/d$.  In the setting of Theorem \ref{simplemaintheorem}, Conjecture \ref{SXconj} therefore predicts that $h^d_{(2)}(Y(\gn)) \ll_\epsilon \textup{vol}(Y(\gn))^{d/(N-1) + \epsilon}$, so that Theorem \ref{simplemaintheorem} in fact represents a strengthening of this conjecture.

We note that there has also been significant progress recently on the problem of showing that the normalized discrete spectral measure of $L^2(\Gamma_n \backslash G)$ tends weakly to the Plancherel measure of $G$.  This work is in some sense orthogonal to ours, and as formulated these results do not provide information on $m(\pi, \Gamma_n)$ beyond showing that $m(\pi, \Gamma_n) / V(n)$ approaches the expected value.

\subsection{Outline of the proof}

It will be more convenient for us to work on the possibly disconnected arithmetic quotients $X(\gn) = G(F) \backslash G(\A) /  K(\gn) K_\infty$.  If $q_v$ denotes the order of the residue field of $F_v$, we prove the following more precise bound.

\begin{theorem}
\label{maintheorem}

Assume the endoscopic classification for inner forms of $U(N)$ stated in Theorem 1.7.1 of \cite{KMSW}.  If $d < N-1$, we have $h^d_{(2)}(X(\gn)) \ll \prod_{v | \gn} (1 - 1/q_v) N\gn^{Nd+1}$, except when $N = 4$ and $d = 2$ when we have $h^d_{(2)}(X(\gn)) \ll \prod_{v | \gn} (1 + 1/q_v) N\gn^{Nd+1}$.

\end{theorem}

Theorem \ref{simplemaintheorem} follows from this, as $X(\gn)$ contains $\gg_\epsilon N\gn^{1-\epsilon}$ copies of $Y(\gn)$ and we have $\text{vol}(Y(\gn)) = N\gn^{N^2-1 + o(1)}$.  We now give an outline of the proof of Theorem \ref{maintheorem}.  For simplicity, we shall either omit or simplify much of the notation for things like Arthur parameters and packets. Because of this, all notation introduced here is temporary. (Refer to Section \ref{notation} below for unexplained notation.)

Let $\Phi_\text{sim}(n)$ denote the set of conjugate self-dual cusp forms on $GL(n,\A_E)$, and let $\nu(l)$ denote the unique irreducible (complex algebraic) representation of $SL(2,\C)$ of dimension $l$.  Let $U(n)$ be the quasi-split unitary group of degree $n$ with respect to $E/F$.  Let $\Psi_2(n)$ denote the set of square-integrable Arthur parameters for $U(n)$, which are formal sums $\psi = \phi_1 \boxtimes \nu(n_1) \boxplus \cdots \boxplus \phi_k \boxtimes \nu(n_k)$ with $\phi_i \in \Phi_\text{sim}(m_i)$, subject to certain conditions including that $n = \sum_{i\ge 1} n_i m_i$ and that the pairs $\phi_i \boxtimes \nu(n_i)$ have to be distinct.  Any $\phi \in \Phi_\text{sim}(n)$ (resp. $\psi \in \Psi_2(n)$) has localizations $\phi_v$ (resp. $\psi_v$), which are local Langlands parameters (resp. Arthur parameters) for $U(n)$.  To each $\psi \in \Psi_2(N)$ and each place $v$ of $F$, there is associated a local packet $\Pi_{\psi_v}(G)$ of representations of $G(F_v)$, and a global packet $\Pi_\psi(G) = \prod \Pi_{\psi_v}(G)$.   If $\psi \in \Psi_2(N)$ and $K \subset G(\A_f)$ we define $\dim_G(K,\psi) = \sum_{\pi \in \Pi_\psi(G)} \dim( \pi_f^K )$.  Similarly, one may associate to $\psi \in \Psi_2(n)$ a packet $\Pi_\psi(U(n)) = \prod \Pi_{\psi_v}(U(n))$ for $U(n,\A)$, and we define $\dim_{U(n)}(K,\psi)$ for $K \subset U(n,\A_f)$ analogously to $\dim_G(K,\psi)$.

The main result of the endoscopic classification implies that the automorphic spectrum of $G$ is contained in the union of $\Pi_\psi(G)$ for $\psi \in \Psi_2(N)$.  If we combine this classification with Matsushima's formula, we have

\be
\label{Matsueg}
h^d_{(2)}(X(\gn)) \le \sum_{\psi \in \Psi_2(N)} \sum_{\pi \in \Pi_\psi(G) } \dim H^d (\g, K_\infty; \pi_\infty) \dim \pi_f^{K(\gn)}.
\ee
The main part of the proof involves using the structure of the packets $\Pi_\psi(G)$ to bound the right hand side of (\ref{Matsueg}) in terms of global multiplicities on smaller quasi-split unitary groups, which we then bound using a theorem of Savin.   The key fact that allows us to control the power of $N\gn$ we obtain is due to Bergeron, Millson, and Moeglin \cite[Prop 13.2]{BMM}, and essentially states that if there exists $\pi \in \Pi_\psi(G)$ with $H^d (\g, K_\infty; \pi_\infty) \neq 0$, then $\psi$ must contain a representation $\nu(n)$ with $n \ge N-d$.

We define a shape to be a list of pairs $(n_1, m_1), \ldots, (n_k, m_k)$ with $\sum_{i\ge 1} m_i n_i = N$, and may naturally talk about the shape of an Arthur parameter.  If $\cS = (n_1, m_1), \ldots, (n_k, m_k)$ is a shape, we let $\Psi_2(N)_\cS \subset \Psi_2(N)$ be the set of parameters having that shape.  If $\psi \in \Psi_2(N)_\cS$, we let $\phi_i \in \Phi_\text{sim}(m_i)$ be the terms in the decomposition $\psi = \phi_1 \boxtimes \nu(n_1) \boxplus \cdots \boxplus \phi_k \boxtimes \nu(n_k)$.  We also define $P_\cS$ to be the standard parabolic in $GL_N$ of type
\[
(\underbrace{m_1, \ldots, m_1}_{n_1 \text{ times}}, \ldots, \underbrace{m_k, \ldots, m_k}_{n_k \text{ times}}).
\]

We now fix $\cS$, and bound the contribution to (\ref{Matsueg}) from $\Psi_2(N)_\cS$, which we denote $h^d_{(2)}(X(\gn))_\cS$.  As mentioned above, we may assume that $n_1 \ge N-d$.  We may restrict our attention to those $\psi \in \Psi_2(N)_\cS$ for which there is $\pi \in \Pi_\psi(G)$ with $H^d (\g, K_\infty; \pi_\infty) \neq 0$.  This condition restricts $\psi_\infty$, and hence $\phi_{i,\infty}$, to finite sets which we denote $\Psi_\infty$ and $\Phi_{i,\infty}$, so that

\bes
h^d_{(2)}(X(\gn))_\cS \ll \sum_{ \substack{ \psi \in \Psi_2(N)_\cS \\ \psi_\infty \in \Psi_\infty } } \dim_G( K(\gn), \psi ).
\ees
In Section \ref{singleparameter} we prove Proposition \ref{singlefinite}, which  states that if the principal congruence subgroups $K_i(\gn) \subset U(m_i,\A_f)$ are chosen correctly, then one can bound $\dim_G( K(\gn), \psi )$ in terms of $\dim_{U(m_i)}( K_i(\gn), \phi_i )$.  For most choices of $\cS$, this bound has the form

\be
\label{outlinesum4}
\dim_G( K(\gn), \psi ) \ll N\gn^{\dim GL_N / P_\cS + \epsilon} \prod_{i=1}^k \dim_{U(m_i)}( K_i(\gn), \phi_i )^{n_i}.
\ee
We prove this bound by factorizing both sides over places of $F$.  At nonsplit places we apply the trace identities that appear in the definition of the local packets $\Pi_{\psi_v}(G)$.  At split places, $\Pi_{\psi_v}(G)$ is a singleton $\{ \pi_v \}$, and we use the description of $\pi_v$ as the Langlands quotient of a representation induced from $P_\cS$.

We next sum the bound (\ref{outlinesum4}) over $\psi \in \Psi_2(N)_\cS$, or equivalently we sum $\phi_i$ over $\Phi_\text{sim}(m_i)$, which gives

\begin{align}
\notag
h^d_{(2)}(X(\gn))_\cS & \ll  N\gn^{\dim GL_N / P_\cS + \epsilon} \prod_{i=1}^k \sum_{\substack{\phi_i \in \Phi_\text{sim}(m_i) \\ \phi_{i, \infty} \in \Phi_{i, \infty} }} \dim_{U(m_i)}( K_i(\gn), \phi_i )^{n_i} \\
\label{outlinesum2}
& \le N\gn^{\dim GL_N / P_\cS + \epsilon} \prod_{i=1}^k \left( \sum_{\substack{\phi_i \in \Phi_\text{sim}(m_i) \\ \phi_{i, \infty} \in \Phi_{i, \infty} }} \dim_{U(m_i)}( K_i(\gn), \phi_i ) \right)^{n_i}.
\end{align}
If we define $\Theta_{i, \infty}$ to be the union of $\Pi_{\phi_{i, \infty}}( U(m_i) )$ over $\phi_{i, \infty} \in \Phi_{i, \infty}$, then $\Theta_{i, \infty}$ is finite.  Moreover, because the parameters $\phi_i$ are simple generic, the packet $\Pi_{\phi_i}( U(m_i))$ is stable, so all representations in it occur discretely on $U(m_i)$.  This implies that

\be
\label{outlinesum3}
\sum_{\substack{\phi_i \in \Phi_\text{sim}(m_i) \\ \phi_{i, \infty} \in \Phi_{i, \infty} }} \dim_{U(m_i)}( K_i(\gn), \phi_i ) \le \sum_{ \pi_\infty \in \Theta_{i, \infty} } m( \pi_\infty, \gn  ),
\ee
where $m( \pi_\infty, \gn  )$ denotes the multiplicity of $\pi_\infty$ in $L^2_\text{disc}( U(m_i,F) \backslash U(m_i,\A) / K_i(\gn) )$.  In fact it follows from the known cases of the Ramanujan conjecture that $\pi_\infty$ is tempered, so $\pi_\infty$ appears only in the cuspidal spectrum. Then a theorem of Savin \cite{Sa} gives $m( \pi_\infty, \gn ) \ll N\gn^{m_i^2}$ for all $\pi_\infty$.  Combining this with (\ref{outlinesum2}) and (\ref{outlinesum3}) gives a bound
\[
h^d_{(2)}(X(\gn))_\cS \ll N\gn^{\dim GL_N / U_\cS + \epsilon}
\]
where $U_\cS$ is the unipotent radical of $P_\cS$.  Showing that $\dim GL_N / U_\cS \le Nd+1$ completes the proof.

The role played by the cohomological degree in this argument is that $\dim U_\cS$ must be large if $d$ is small, because of the bound $n_1 \ge N-d$.  However, it should be noted that the bound $\dim GL_N / U_\cS \le Nd+1$ does not need to hold if $m_1 \le 3$, and in these cases there are some additional steps one must take to optimize the argument to obtain the exponent $Nd+1$.  We will describe them in the course of the proof in the main body except for the following key input, which may be of independent interest. Namely we give in Theorem \ref{invdim} a uniform bound (which is significantly better than a trivial bound; see the remark below Corollary \ref{corinvdim}) on the dimension of invariant vectors in supercuspidal representations of $GL(3)$ under principal congruence subgroups.\footnote{The same problem for $GL(2)$ also occurs, but this is easier and handled differently based on a construction of supercuspidals via Weil representations. See Section \ref{boundGL2}. As we explained the exceptional case occurs only when $m_1\le 3$, so we need not consider the problem for $GL(n)$ with $n>3$.} By a uniform bound we mean a bound which is independent of the representation (and only depends on the residue field cardinality and the level of congruence subgroup). The asymptotic growth of the invariant dimension is fairly well understood if a representation is fixed but not otherwise.  Analogous uniform bounds, on which our paper sheds some light, should be useful for bounding the growth of cohomology of other locally symmetric spaces.

Our argument in fact shows that $h_{(2)}^d(X(\gn))_\cS \ll_\epsilon N\gn^{Nd +\epsilon}$, except when $\cS = (N-d,1), (1,d)$, or in the exceptional case when $N = 4$, $d = 2$, and $\cS = (2,2)$.  Moreover, when $\cS = (N-d,1), (1,d)$ we expect that the bound $h_{(2)}^d(X(\gn))_\cS \ll_\epsilon N\gn^{Nd + 1 + \epsilon}$ is sharp when $G$ arises from a Hermitian form, so that the majority of $H_{(2)}^d(X(\gn))$ comes from parameters of this shape.  By \cite[Theorem 10.1]{BMM}, these forms are theta lifted from a Hermitian space of dimension $d$, and it may therefore be possible to prove that Theorem \ref{maintheorem} is sharp using the theta lift.

\subsection{Acknowledgements}

We would like to thank Nicolas Bergeron for helpful discussions.  The first author was supported by the NSF under Grant No. 1440140 while in residence at the MSRI in Berkeley, California, during the spring of 2017.

\section{Notation}\label{notation}

Our notation and discussion in this section are based on \cite{M} and \cite{KMSW}. (Similar summaries are given in \cite{Ma2} and \cite{Ma1} with more details in the quasi-split case.)

Let $N$ be a positive integer. Write $GL(N)$ for the general linear group. Let $F$ be a field of characteristic zero. Given a quadratic algebra $E$ over $F$, we define $U(N)=U_{E/F}(N)$ to be the quasi-split unitary group in $N$ variables, defined by an antidiagonal matrix $J_N$ with $(-1)^{i-1}$ in the $(i,N+1-i)$ entry, as in \cite[0.2.2]{KMSW}. The compact special unitary group in two variables is denoted by $SU(2)$. Let $\nu(n)$ denote its $n$-dimensional irreducible representation (unique up to isomorphism).

Assume that $F$ is a local or global field of characteristic zero. Write $W_F$ for the Weil group of $F$. For any connected reductive group $G$ over $F$, its Langlands dual group is denoted by $\widehat G$. Let $^L G=\widehat G \rtimes W_F$ denote the (Weil form of) $L$-group of $G$. Note that $^L GL(N)=GL(N,\C)\times W_F$ and that $^L U_{E/F}(N)$ may be explicitly described, cf. \cite[0.2.2]{KMSW}.

Now assume that $F$ is local.
Define the local Langlands group $L_F:=W_F$ if $F$ is archimedean and $L_F:=W_F\times SU(2)$ otherwise. An $A$-parameter is a continuous homomorphism $\psi : L_F \times SL(2,\C) \rightarrow {}^L G$ commuting with the projection maps onto $W_F$ such that $\psi(L_F)$ has relatively compact image in $\widehat G$ and that $\psi$ restricted to $SL(2,\C)$ is a map of $\C$-algebraic groups into $\widehat G$. Two parameters are considered isomorphic if they are conjugate under $\widehat G$. Write $\Psi(G)$ or $\Psi(G,F)$ for the set of isomorphism classes of $A$-parameters. Define $\Psi^+(G)$ analogously without the condition on relatively compact image. Define $s_\psi:=\psi(1,-1)$ for any $\psi\in \Psi^+(G)$.

An $L$-parameter is $\psi^+\in \Psi^+(G)$ which is trivial on the $SL(2,\C)$-factor (external to $L_F$). The subset of $L$-parameters (up to isomorphism) is denoted by $\Phi(G)$. Any $\psi\in \Psi^+(G)$ gives rise to an $L$-parameter $\phi_\psi$ by pulling back via  the map $L_F \to L_F \times SL(2,\C)$, $w \mapsto\left( w, \left( \begin{smallmatrix} |w|^{1/2} & 0 \\ 0 & |w|^{-1/2} \end{smallmatrix} \right) \right)$.

When $G=GL(N)$, we associate representations $\pi_\psi$ and $\rho_\psi$ of $\GL(N,F)$ to $\psi\in \Psi^+(G)$ as follows:
we may decompose $\psi=\oplus_{i=1}^k \psi_i$ with $\psi_i=\phi_i\boxtimes \nu(n_i)$ such that $\phi_i: L_F \to {}^L GL(m_i)$ is an irreducible $m_i$-dimensional representation of $L_F$ and $\sum_{i=1}^k m_i n_i = N$. The local Langlands correspondence associates an irreducible representation $\pi_{\phi_i}$ of $GL(m_i,F)$ to $\phi_i$. Let $|\det(m)|$ denote the composition of the absolute value on $F^\times$ with the determinant map on $GL(m,F)$. Then consider the multi-set of representations
\be\label{langlandsquot}
\{\pi_{\phi_i}|\det(m_i)|^{\frac{n_i-1}{2}}, \pi_{\phi_i}|\det(m_i)|^{\frac{n_i-3}{2}}, ..., \pi_{\phi_i}|\det(m_i)|^{\frac{1-n_i}{2}}\}_{i=1}^k.
\ee
This defines a representation of $\prod_{i=1}^k GL(m_i)^{n_i}$ viewed as a block diagonal Levi subgroup of $GL(N)$. Let $\rho_{\psi}$ denote the parabolically induced representation. (The choice of parabolic subgroup does not affect our argument; we will choose the upper triangular one.) The Langlands quotient construction singles out an irreducible subquotient $\pi_{\psi}$ of $\rho_{\psi}$ by concatenating the Langlands quotient data for the representations in \eqref{langlandsquot}. More concisely, $\pi_{\psi}$ corresponds to the parameter $\phi_\psi$ obtained from $\psi$ as above.

As in the introduction, from here throughout the paper, we fix a totally real field $F$ and a totally complex quadratic extension $E$ over $F$ with complex conjugation $c$ in $\mathrm{Gal}(E/F)$. The ring of adeles over $F$ (resp. $E$) is denoted by $\A$ (resp. $\A_E$). We often write $G^*$ for $U(N)$ and $G(N)$ for $\mathrm{Res}_{E/F}GL(N)$. The group $G(N)$ is equipped with involution $\theta: g\mapsto J_N {}^t c(g)^{-1} J_N^{-1}$, giving rise to the twisted group $\widetilde G^+(N)=G(N)\rtimes \{1,\theta\}$. Write $\widetilde G(N)$ for the coset $G(N)\rtimes \theta$. Let $v$ be a place of $F$. Given any algebraic group $H$ over $F$, we often write $H_v$ for $H(F_v)$ or $H \otimes_F F_v$ (the context will make it clear which one we mean).

For $n\in \Z_{\ge 1}$, define $\widetilde\Phi_{\mathrm{sim}}(n)$ (a shorthand for $\Phi_{\mathrm{sim}}(\widetilde G(n))$) to be the set of conjugate self-dual cuspidal automorphic representations of $GL(n,\A_E)$. Here a representation $\pi$ is considered conjugate self-dual if $\pi\circ c$ is isomorphic to the contragredient of $\pi$, or equivalently if $\pi\circ \theta$ is isomorphic to $\pi$. Fix two Hecke characters $\chi_\kappa: \A^\times_E/E^\times \to \C^\times$ with $\kappa\in \{\pm 1\}$ as follows: $\chi_+$ is the trivial character while $\chi_-$ is an extension of the quadratic character of $\A^\times/F^\times$ associated to $E/F$ by class field theory. We use $\chi_\kappa$ to define two base-change $L$-morphisms
$$\eta_{\chi_\kappa}: {}^L U(N) \to {}^L G(N),\quad \kappa\in \{\pm 1\},$$
as follows. Choose $w_c\in W_F \backslash W_E$ so that $W_F=W_E \coprod W_E w_c$. Under the identification $\widehat U(N)=GL(N,\C)$ and $\widehat G(N) = GL(N,\C)\times GL(N,\C)$, we have (where scalars stand for scalar $N\times N$-matrices whenever appropriate)
\begin{align}
\notag
\eta_{\chi_\kappa}(g\rtimes 1)  & = (g,J_N {}^t g^{-1} J_N^{-1})\rtimes 1, \\
\notag
\eta_{\chi_\kappa}(1\rtimes w)  & = (\chi_\kappa(w),\chi_\kappa^{-1}(w))\rtimes w, \quad w \in W_E,\\
\notag
\eta_{\chi_\kappa}(1\rtimes w_c)  & = (1,\kappa)\rtimes w_c.
\end{align}

Let us define $\widetilde\Psi_{\mathrm{ell}}(N)$, the set of (formal) elliptic parameters for $\widetilde G(N)$. Such a parameter is represented by a formal sum $\psi=\boxplus_{i=1}^k \psi_i$ with $\psi_i=\mu_i\boxtimes \nu(n_i)$ such that the pairs $(\mu_i,n_i)$ are mutually distinct, where $\mu_i\in \widetilde\Phi_{\mathrm{sim}}(m_i)$, $\sum_{i=1}^k m_i n_i = N$. (Two formal sums are identified under permutation of indices.)

Mok defines the sets $\Psi_2(U(N),\eta_{\chi_\kappa})$ for $\kappa\in \{\pm 1\}$. (In \cite{M}, he writes $\xi_{\chi_\kappa}$ for $ \eta_{\chi_\kappa}$.) They are identified (via the map $(\psi^N,\widetilde\psi)\mapsto \psi^N$ of \cite[Section 2.4]{M}) with disjoint subsets of $\widetilde\Psi_{\mathrm{ell}}(N)$, corresponding to the two ways $U(N)$ can be viewed as a twisted endoscopic group of $\widetilde G(N)$ via $\eta_{\chi_\kappa}$, characterized by a sign condition. We don't need to recall the sign condition here. It suffices to know that each $\psi\in \Psi_2(U(N),\eta_{\chi_\kappa})$ admits localizations to $\Psi^+(U(N)_v)$; see below. We write $\psi^N$ for $\psi$ when $\psi$ is viewed as a member of $\widetilde\Psi_{\mathrm{ell}}(N)$.

A parameter $\psi$ in $\Psi_2(U(N),\eta_{\chi_\kappa})$ is said to be generic if $n_i=1$ for all $1\le i\le k$ and simple if $k=1$. Write $\Phi_{\mathrm{sim}}(U(N),\eta_{\chi_\kappa})$ for the subset of simple generic parameters. Theorem 2.4.2 of \cite{M} shows that $\widetilde\Phi_{\mathrm{sim}}(N)$ is partitioned into $\Phi_{\mathrm{sim}}(U(N),\eta_{\chi_\kappa})$, $\kappa\in \{\pm1\}$.

To a parameter $\psi\in \Psi_2(U(N),\eta_{\chi_\kappa})$ is associated localizations $\psi_v\in \Psi^+(G^*_v)$ such that $ \psi_v$ is carried to $\oplus_{i=1}^k \phi_{\mu_{i,v}}\boxtimes \nu(n_i)$ via the $L$-morphism $\eta_{\chi_\kappa}: {}^L U(N)\to {}^L G(N)$, where $\phi_{\mu_{i,v}}$ is the $L$-parameter for $\mu_{i,v}$ (via local Langlands for $GL(m_i)$). For each place $v$ of $F$ split in $E$, fix a place $w$ of $E$ above $v$. Then we have an isomorphism $G^*_v\simeq GL(N,E_w)$.

 At every finite place $v$ of $F$ where $G^*_v$ is unramified, fix hyperspecial subgroups $\widetilde K_v=GL(\cO_{F_v}\otimes_{\cO_F} \cO_E)$ of $G(N,F_v)$ and  $K^*_v$ of $G^*(F_v)$ (such that they come from global integral models away from finitely many $v$). When $v$ is split as $w$ and $c(w)$ in $E$ we have a decomposition $\widetilde K_v = \widetilde K_w \times \widetilde K_{c(w)}$, and we may identify $K^*_v$ with $\widetilde K_w$ via $G^*_v\simeq GL(N,E_w)$.

  Finally let $G$ be an inner form of $G^*$ over $F$. It can always be promoted to an extended pure inner twist $(\xi,z):G^* \to G$, \cite[0.3.3]{KMSW}. Let $S$ be a set of places of $F$ such that both $G_v$ and $G^*_v$ are unramified for every $v\notin S$. Then fix an isomorphism $G^*_v\simeq G_v$, which is $G(\overline F)$-conjugate to $(\xi,z)$. We have a hyperspecial subgroup $K_v\subset G_v$ by transferring $K^*_v$. So if $v \notin S$ is split in $E$ then $K_v$ and $K^*_v$ are identified with $GL(N,\cO_{E,w})$ under the isomorphisms $G_v\simeq G^*_v \simeq GL(N,E_w)$.

  Let $\psi_v\in \Psi(U(N)_v)$ for a place $v$ of $F$. This gives rise to a distribution $f\mapsto f(\psi_v)$ on the space of smooth compactly supported functions on $U(N)_v$ \cite[Theorem 3.2.1]{M}.

  Given a connected reductive group $H$ over $F_v$, a smooth compactly supported function $f$ on $H(F_v)$, and an admissible representation $\pi$ of $H(F_v)$, we write $\mathrm{tr}(\pi(f))$ or $f(\pi)$ for the trace value. Occasionally we also consider a twisted variant when $\tilde\pi$ is an admissible representation of $G^+(N,F_v)$ and $\tilde f$ is a smooth compactly supported function on $G(N,F_v)\rtimes \theta$. Then $\mathrm{tr}(\tilde\pi(f))$ will denote the (twisted) trace.


\section{Cohomological representations of $U(N-1,1)$}\label{cohomologicalrep}

In this section, we recall some facts about the cohomological representations of the real Lie group $U(N-1,1)$, which will imply that any global Arthur parameter that contributes to $h^d_{(2)}(X(\gn))$ must have a factor $\mu \boxtimes \nu(n)$ with $n \ge N-d$ by applying results of Bergeron, Millson, and Moeglin.  Let $\g_0$ be the real Lie algebra of $U(N-1,1)$, and $K$ a maximal compact subgroup. Write $\g$ for the complexification of $\g_0$. Similarly the complexification of real Lie algebras $\gk_0$, $\p_0$, etc will be denoted by $\gk$, $\p$, etc below. The facts we shall need on the cohomological representations of $U(N-1,1)$ are summarized in the following proposition; recall that $p(\pi)$ is the infimum over $p$ for which the $K$-finite matrix coefficients of $\pi$ lie in $L^p(G)$.

\begin{prop}
\label{cohreps}

Let $a,b$ be a pair of integers with $a, b \ge 0$ and $a+b \le N-1$, and let $d = a+b$.  There is an irreducible unitary representation $\pi_{a,b}$ of $U(N-1,1)$ with the following properties.

\begin{enumerate}

\item
\label{piab1}
We have

\[
H^{p,q}(\g,K; \pi_{a,b}) = \bigg\{ \begin{array}{ll} \C & \text{if } (p,q) = (a,b) + (k,k), \quad 0 \le k \le N-1-a-b, \\
0 & \text{otherwise}. \end{array}
\]

\item
\label{piab2}
Suppose that $d \le N-2$.  If $\varphi: \C^\times \to GL(N,\C)$ is the restriction of the Langlands parameter of $\pi_{a,b}$ to $\C^\times$, then we have
\[
\varphi(z) = (z / \overline{z})^{(b-a)/2} |z|^{N-d-1} \oplus (z / \overline{z})^{(b-a)/2} |z|^{-N+d+1} \oplus \bigoplus_{ \substack{ -N+1 \le j \le N-1 \\ j \equiv N-1 \; (2) \\ j \neq N-1-2a, -N+1+2b } } (z / \overline{z})^{j/2}.
\]

\item
\label{piab3}
We have $p(\pi_{a,b}) = 2(N-1)/d$.

\end{enumerate}
Moreover, the $\pi_{a,b}$ are the only irreducible unitary representations of $U(N-1,1)$ with $H^*(\g,K; \pi) \neq 0$.

\end{prop}

\subsection{The classification of Vogan and Zuckerman}

We let $G = U(N-1,1)$, and realize $G$ as the subgroup of $GL(N,\C)$ preserving the Hermitian form $|z_1|^2 + \ldots + |z_{N-1}|^2 - |z_N|^2$.  The Lie algebra $\g_0$ of $G$ is
\[
\g_0 = \{ A \in M_N(\C) : {}^t \overline{A} = - I_{N-1,1} A I_{N-1,1} \}
\]
where
\[
I_{N-1,1} = \left( \begin{array}{cc} I_{N-1} & \\ & -1 \end{array} \right).
\]
The algebras $\gk_0$ and $\p_0$ in the Cartan decomposition $\g_0 = \gk_0 \oplus \p_0$ are
\[
\gk_0 = \left\{ \left( \begin{array}{cc} A & 0 \\ 0 & i\theta \end{array} \right) : {}^t \overline{A} = -A, \theta \in \R \right\}, \quad \p_0 = \left\{ \left( \begin{array}{cc} 0 & z \\ {}^t\overline{z} & 0 \end{array} \right) : z \in M_{N-1,1}(\C) \right\}.
\]
Let $\gt_0$ denote the Cartan subalgebra of $\gk_0$ consisting of diagonal matrices.  The adjoint action of $K$ on $\p_0$ preserves the natural complex structure, and so we have a decomposition $\p = \p_+ \oplus \p_-$ of $K$-modules.  We may naturally identify $\g$ with $M_n(\C)$, and under this identification we have
\[
\p_+ = \left\{ \left( \begin{array}{cc} 0 & z \\ 0 & 0 \end{array} \right) : z \in M_{N-1,1}(\C) \right\}, \quad \p_- = \left\{ \left( \begin{array}{cc} 0 & 0 \\ z & 0 \end{array} \right) : z \in M_{1,N-1}(\C) \right\}.
\]
If $\tau_d$ is the representation of $K$ on $\bigwedge^d \p$, it is well known \cite[VI 4.8-9]{BW} that there is a decomposition
\be
\label{taud}
\tau_d = \oplus_{a+b = d} \tau_{a,b},
\ee
where $\tau_{a,b}$ is the representation of $K$ on $\bigwedge^a \p_- \otimes \bigwedge^b \p_+$.  Moreover, we have
\be
\label{tauab}
\tau_{a,b} = \oplus_{k=0}^{\min(a,b)} \tau'_{a-k,b-k}
\ee
for $a+b \le N-1$, where the representations $\tau'_{a,b}$ are irreducible with highest weight

\be
\label{tauweight}
\sum_{i=1}^b \varepsilon_i - \sum_{i = N-a}^{N-1} \varepsilon_i + (a-b)\varepsilon_N.
\ee
Here, $\{ \varepsilon_i \}$ is the standard basis for $\gt^*$ consisting of elements that are real on $i \gt_0$.  These decompositions correspond to the Hodge-Lefschetz decomposition for the cohomology of $X(\gn)$.

We now recall the classification of cohomological representations of $G$ due to Vogan and Zuckerman \cite{VZ}.  We choose an element $H \in i\gt_0$, so that $\ad(H)$ has real eigenvalues.  We let $\gq \subset \g$ be the parabolic subalgebra $\gl + \gu$, where $\gl = Z_\g(H)$ and $\gu$ is the sum of all the eigenspaces for $\ad(H)$ with positive eigenvalues.  Because $\gk$ and $\p_\pm$ are stable under $\ad(H)$, we have $\gu = \gu \cap \gk + \gu \cap \p_- + \gu \cap \p_+$.  We define $R_\pm = \dim( \gu \cap \p_\pm)$ and $R = R_+ + R_-$, and let $\mu = 2\rho(\gu \cap \p)$, which is the sum of the roots of $\gt$ in $\gu \cap \p$. We fix a set of positive roots for $\gt$ in $\gl\cap \gk$ so that a positive root system for $\gt$ in $\gk$ is determined (together with $\gu\cap \gk$). Then $\mu$ is a highest weight for the positive root system.

The main theorem of Vogan and Zuckerman is that there is a unique irreducible unitary representation $A_\gq$ of $G$ \footnote{Note that the general unitarity of the representations $A_\gq$ is proved in \cite{V}.}  with the following properties:

\begin{itemize}

\item $A_\gq$ has the same infinitesimal character as the trivial representation.

\item $A_\gq$ contains the $K$-type with highest weight $\mu$.

\end{itemize}

They also show that any irreducible unitary representation of $G$ with nonzero $(\g,K)$-cohomology (with trivial coefficients) must be of the form $A_\gq$ for some $\gq$.  It is clear that $A_\gq$ only depends on $\gu \cap \p$.  Moreover, we have \cite[Prop 6.19]{VZ}
\be
\label{AqHodge1}
H^{R_+ + p, R_- + p}(\g,K; A_\gq) \simeq \text{Hom}_{\gl \cap \gk}( \wedge^{2p}(\gl \cap \p), \C),\quad p\ge 0,
\ee
and
\be
\label{AqHodge2}
H^{p, q}(\g,K; A_\gq) = 0
\ee
for other $(p,q)$, i.e. if $p-q\neq R_+-R_-$.

Write $H_1, \ldots, H_N$ for the entries of the real diagonal matrix $H$.  Because $A_\gq$ only depends on the orbit of $H$ under the Weyl group of $K$, we may assume that $H_1 \ge \cdots \ge H_{N-1}$.  The subspace $\gu \cap \p$, and hence $A_\gq$, only depends on the number of $H_i - H_N$ that are positive, negative, and zero.  Therefore, if $a$ and $b$ are the number of $H_i - H_N$ that are positive and negative respectively, then we have $H_{a+1} = \cdots = H_{N-1-b} = H_N$, while we may assume that all the remaining $H_i$ are distinct.  It may be seen that $R_+ = a$ and $R_- = b$, and
\[
\mu = \sum_{i = 1}^a \varepsilon_i - \sum_{i = N-b}^{N-1} \varepsilon _i - (a-b) \varepsilon_N.
\]
The representation $A_\gq$ depends only on $a$ and $b$, and we denote it by $\pi_{a,b}$.

To prove (\ref{piab3}), we will need the description of $\pi_{a,b}$ as a Langlands quotient when $a+b < N-1$, which is given by Vogan and Zuckerman in \cite[Theorem 6.16]{VZ}.  Define
\[
V = \left( \begin{array}{ccc} && 1 \\ & 0_{N-2} & \\ 1 && \end{array} \right),
\]
and let $\ga_0 = \R V$ so that $\ga_0$ is a maximal abelian subalgebra of $\p_0$.  Let $A = \exp(\ga_0)$ be the corresponding subgroup.  Define $\alpha \in \ga^*$ by $\alpha(V) = 1$.  The roots of $\ga$ in $\g$ are $\pm \alpha$ and $\pm 2\alpha$ with multiplicities $2(N-2)$ and 1 respectively, so that $\rho = (N-1)\alpha$.  Let $U$ be the unipotent subgroup corresponding to the positive roots.  Let $M = Z_K(V)$, so that
\[
M = \left\{ \left( \begin{array}{ccc} e^{i\theta} && \\ & X & \\ && e^{i\theta} \end{array} \right) : X \in U(N-2), \theta \in \R \right\}.
\]
Let $\gt_M \subset \gt$ be the diagonal Cartan subalgebra in $\gm$.  Let $\sigma$ be the irreducible representation of $M$ with highest weight given by the restriction to $\gt_M$ of
\[
\sum_{i=2}^{a+1} \varepsilon_i - \sum_{i= N-b}^{N-1} \varepsilon_i + (b-a) \varepsilon_1.
\]
Let $\nu = (N-1-d)\alpha$.  We define $I_{\nu,\sigma}$ to be the unitarily normalized induction from $P = MAU$ to $G$ of the representation $\sigma \otimes e^\nu \otimes 1$.  Then $\pi_{a,b}$ is the Langlands quotient of $I_{\nu,\sigma}$.

\subsection{Proof of Proposition \ref{cohreps} }

The assertion that $\pi_{a,b}$ are the only representations with nonzero cohomology is clear, because any such representation is isomorphic to $A_\gq$ for some $\gq$.  The calculation of $H^{p,q}(\g,K; \pi_{a,b})$ in condition (\ref{piab1}) follows from (\ref{AqHodge1}) and (\ref{AqHodge2}) after we compute $\text{Hom}_{\gl \cap \gk}( \wedge^{2p}(\gl \cap \p), \C)$.  Our assumption on $H$ implies that $\gl_0 \simeq \gu(N-d-1,1) \times \gu(1)^d$, and $\gl = \gl \cap \gk \oplus \gl \cap \p$ is the standard Cartan decomposition of $\gl$.  We wish to show that the trivial representation of $\gl \cap \gk$ occurs exactly once in $\wedge^{2p}(\gl \cap \p)$ for all $0 \le p \le N-d-1$, but this follows from the decompositions (\ref{taud}) and (\ref{tauab}) for $\gu(N-d-1,1)$, combined with the fact that $\tau'_{p,q}$ is trivial if and only if $p=q=0$ as one sees from the highest weight formula (\ref{tauweight}).

The description of the Langlands parameter of $\pi_{a,b}$ in (\ref{piab2}) follows from \cite[Section 5.3]{BC}.

To prove assertion (\ref{piab3}), we may assume that $a+b < N-1$ as otherwise $\pi_{a,b}$ lies in the discrete series.  When $a+b < N-1$, the assertion follows from our description of $\pi_{a,b}$ as a Langlands quotient, and well-known asymptotics for matrix coefficients, which we recall from Knapp \cite{Kn}.  Let $\overline{P} = MA \overline{U}$ be the opposite parabolic to $P$, and let $\overline{I}_{\sigma,\nu}$ be the normalized induction of $\sigma \otimes e^\nu \otimes 1$ from $\overline{P}$ to $G$.  Let $A(\sigma,\nu): I_{\sigma,\nu} \to \overline{I}_{\sigma,\nu}$ be the intertwiner
\[
A(\sigma,\nu)f(g) = \int_{\overline{U}} f(ug) du,
\]
which converges by \cite[VII, Prop 7.8]{Kn}.  Then the image of $A(\sigma,\nu)$ is isomorphic to the Langlands quotient $\pi_{a,b}$ of $I_{\sigma,\nu}$.  We introduce the pairing on $I_{\sigma,\nu}$ given by
\[
\langle f,g \rangle = \int_K \langle f(k), g(k) \rangle_\sigma dk
\]
where $\langle \cdot, \cdot \rangle_\sigma$ denotes a choice of inner product on $\sigma$.  If we choose $g \in I_{\sigma,\nu}$ to pair trivially with the kernel of $A(\sigma,\nu)$, then $\langle I_{\sigma,\nu}( \cdot ) f, g \rangle$ is a matrix coefficient of $\pi_{a,b}$, and all coefficients are realized in this way.  The asymptotic behaviour of the coefficients is given by \cite[VII, Lemma 7.23]{Kn}, which states that
\be
\label{coeffasymp}
\underset{a \to \infty}{\lim} e^{(\rho-\nu) \log a} \langle I_{\sigma,\nu}(a)f, g \rangle = \langle A(\sigma,\nu)f(1), g(1) \rangle_\sigma.
\ee
As $\nu = (N-d-1)\alpha$, \cite[VIII, Theorem 8.48]{Kn} implies that $p(\pi_{a,b}) \le 2(N-1)/d$.  It also follows from that theorem that to prove $p(\pi_{a,b}) = 2(N-1)/d$, we need only show that the right hand side of (\ref{coeffasymp}) is nonzero for some choice of $f$ and $g$, subject to the condition that $g$ pairs trivially with $\ker A(\sigma,\nu)$.  To do this, choose $f \in I_{\sigma,\nu}$ such that $A(\sigma,\nu)f \neq 0$, and some nonzero $g$ of the required type.  Because $A(\sigma,\nu)$ is an intertwiner, after translating $f$ by $K$ we may assume that $A(\sigma,\nu)f(1) \neq 0$.  Because $\ker A(\sigma,\nu)$ is an invariant subspace, we may likewise assume that $g(1) \neq 0$.  Because $\sigma$ was irreducible, translating by $M$ we may also assume that $\langle A(\sigma,\nu)f(1), g(1) \rangle_\sigma \neq 0$ as required.

\section{Application of the global classification}\label{application}

  As in the notation section, $(\xi,z):G^*\to G$ is an extended pure inner twist of the quasi-split unitary group $G^*=U(N)$ over $F$. We always assume that $G_{v_0}$ is isomorphic to $U(N-1,1)$ at a real place $v_0$ of $F$ and that $G_v$ is compact at all other real places $v$. Although much of our argument works for general inner forms, the assumption significantly simplifies some combinatorial and representation-theoretic arguments (especially of Section \ref{cohomologicalrep}) and ensures that we obtain expectedly optimal upper bounds in all degrees in the main theorem.

 Let $\psi \in \Psi_2(G^*, \eta_{\chi_\kappa} )$. The main local theorem of \cite{KMSW} defines local packets $\Pi_{\psi_v}(G,\xi)$ consisting of finitely many (possibly reducible and non-unitary\footnote{The issue is that $\psi_v\in \Psi^+(G^*_v)$ is not known to be in $\Psi(G^*_v)$ in general although it is expected. However this is actually known for parameters contributing to cohomology from the known cases of the Ramanujan conjecture, see Section \ref{Archcontrol}.  It follows that all representations in the local packets we will consider are irreducible and unitary.}) representations of $G_v$ such that $\Pi_{\psi_v}(G,\xi)$ contains an unramified representation (relative to $K_v$) at all but finitely many $v$. The global packet $\Pi_{\psi}(G,\xi)$ consists of restricted tensor products $\pi=\otimes'_v \pi_v$ with $\pi_v\in \Pi_{\psi_v}(G,\xi)$. The parameter $\psi$ determines a sign character $\epsilon_\psi$ on a certain centralizer group (in $\widehat G$) attached to $\psi$, and \cite{KMSW} defines a subset
 $\Pi_\psi(G, \xi, \epsilon_\psi)$ of $\Pi_\psi(G, \xi)$ by imposing a sign condition. We need not recall the condition as it will be soon ignored along the way to an upper bound. Theorem 1.7.1 of \cite{KMSW} asserts the following.

\begin{theorem}
\label{endoscopy}

There is a $G(\A)$-module isomorphism

\bes
L^2_\textup{disc}( G(F) \backslash G(\A) ) \simeq \bigoplus_{\psi \in \Psi_2(G^*,  \eta_{\chi_\kappa} ) } \bigoplus_{\pi \in \Pi_\psi(G, \xi, \epsilon_\psi) } \pi.
\ees

\end{theorem}

Let $S$ be a finite set of finite places of $F$ containing all places above 2 and 3, and all places at which $E$ or $G$ ramify.  Let $\gn\subset \cO_F$ be a nonzero ideal whose prime factors are split in $E$ and don't lie in $S$. In Section \ref{notation} we have introduced hyperspecial subgroups $K_v$ of $G(F_v)$ when $v\notin S$. For $v\in S$ let $K_v$ be an arbitrary open compact subgroup of $G(F_v)$. Now we define the congruence subgroup $K(\gn)=\prod_v K(\gn)_v$, where $K(\gn)_v$ is given as follows for each finite place $v$. Define $K(\gn)_v$ to be $K_v$ if $v$ does not divide $\gn$. If $v|\gn$ then we have fixed an isomorphism $K_v\simeq GL(N,\cO_{E_w})$, and $K(\gn)_v$ is the subgroup of $K_v$ consisting of elements congruent to the identity modulo $\gn$. Let $K_\infty$ denote a maximal compact subgroup of $G(F\otimes_\Q \R)$.
Often we write $[G]$ for the quotient $G(F)\backslash G(\A)$, and likewise when $G$ is replaced with quasi-split unitary groups.

We would like to investigate the cohomology of the arithmetic manifold
\be\label{X(n)}
X(\gn) = G(F)\backslash G(\A) / K(\gn) K_\infty.
\ee
Since $G(F\otimes_\Q \R)/K_\infty$ is isomorphic to the symmetric space $U(N-1,1)/(U(N-1)\times U(1))$, which has complex dimension $N-1$, we see that the complex dimension of $X(\gn)$ is also $N-1$.

We take the first step in proving Theorem \ref{maintheorem} on bounding the $L^2$-Lefschetz numbers $h^d_{(2)}(X(\gn))$ in degrees $0\le d<N-1$, as $\gn$ varies.
Matsushima's formula gives

\bes
h^d_{(2)}(X(\gn)) = \sum_{\pi \in L^2_\text{disc}([G]) } m(\pi) h^d( \g, K_\infty; \pi_\infty) \dim \pi_f^{K(\gn)}
\ees
(see \cite{BW} in the noncompact case), and combining this with Theorem \ref{endoscopy} gives

\be\label{firststep}
h^d_{(2)}(X(\gn)) \le \sum_{\psi \in \Psi_2(G^*, \eta_{\chi_\kappa} ) } \sum_{\pi \in \Pi_\psi(G, \xi) } h^d( \g, K_\infty; \pi_\infty) \dim \pi_f^{K(\gn)}.
\ee

\section{Bounding the contribution of a single parameter}\label{singleparameter}

In this section, we bound the contribution of a single parameter $\psi$ to the right hand side of (\ref{firststep}).  The form of our bound will depend on the shape of $\psi$, and so throughout this section we shall fix a shape $\cS = (n_1,m_1), \ldots, (n_k,m_k)$ and define $\Psi_2(G^*,\eta_{\chi_\kappa})_\cS$ to be the set of parameters with that shape.  If $\psi \in \Psi_2(G^*,\eta_{\chi_\kappa})_\cS$, we define $\mu_i \in \widetilde{\Phi}_\text{sim}(m_i)$ to be such that $\psi^N = \boxplus_{i \ge 1} \mu_i \boxtimes \nu(n_i)$.  Each $\mu_i$ represents a simple generic parameter $\phi_i \in \Phi_{\mathrm{sim}} ( U(m_i), \eta_{\chi_{\kappa_i}} )$ for a unique sign $\kappa_i\in \{\pm 1\}$ determined as in \cite[(2.4.8)]{M}.  We define

\begin{align}
\label{taudef}
\tau(\cS) & = \binom{N}{2} - \sum_{i \ge 1} n_i \binom{m_i}{2}, \\
\notag
\tau_1(\cS) & = \binom{N}{2} - \binom{n_1}{2} - \sum_{i \ge 2} n_i \binom{m_i}{2}, \\
\notag
\tau_2(\cS) & = \tau(\cS) + (n_1-1), \\
\notag
\tau_3(\cS) & = \tau(\cS) + 4(n_1 - 1) + \epsilon.
\end{align}
Here, $\epsilon > 0$ is an arbitrarily small constant that may vary from line to line.  Any implied constants in bounds for quantities containing $\tau_3(\cS)$ will be assumed to depend on $\epsilon$.  We also define

\bes
\sigma(\cS) = \sigma_3(\cS) = \sum_{i=1}^k n_i - 1, \quad \sigma_1(\cS) = \sum_{i=2}^k n_i, \quad \sigma_2(\cS) = 2(n_1 - 1) + \sum_{i=2}^k n_i.
\ees

For each $1 \le i \le k$ and finite place $v$, define a compact open subgroup $K_{i,v}$ of $U(m_i)_v$ as follows.  If $v \notin S$, then $K_{i,v}$ is the standard hyperspecial subgroup, and if $v \in S$ then $K_{i,v}$ will be chosen during the proof of Proposition \ref{singlefinite}.  Let $K_i = \prod_v K_{i,v}$, and let $K_i(\gn)$ be the principal congruence subgroup of $K_i$ of level $\gn$. Let $P \subset GL(N)$ be the standard parabolic subgroup with Levi $\prod_{i=1}^k GL(m_i)^{n_i}$.

\begin{prop}
\label{singlefinite}

There is a choice of $K_{i,v}$ for $v \in S$ with the following property.  Let $\psi\in \Psi_2(G^*,\eta_{\chi_\kappa})_\cS$, and assume that $\phi_i$ (arising from $\psi$ as above) is bounded everywhere for each $1\le i\le k$. Then

\be
\label{packetsum}
\sum_{\pi \in \Pi_\psi(G, \xi)} \dim \pi_f^{K(\gn)} \ll \prod_{v | \gn} (1 + 1/q_v)^{\sigma(\cS)} N\gn^{\tau(\cS)} \prod_{i\ge 1} \left( \sum_{ \pi_i \in \Pi_{\phi_i}(U(m_i)) } \dim \pi_{i,f}^{K_i(\gn)} \right)^{n_i}.
\ee
Moreover, if $m_1 = l$ with $l = 1, 2, 3$, we have

\bes
\sum_{\pi \in \Pi_\psi(G, \xi)} \dim \pi_f^{K(\gn)} \ll \prod_{v | \gn} (1 + 1/q_v)^{\sigma_l(\cS)} N\gn^{\tau_l(\cS)} \sum_{ \pi_1 \in \Pi_{\phi_1}(U(m_1)) } \dim \pi_{1,f}^{K_1(\gn)}
\prod_{i \ge 2} \left( \sum_{ \pi_i \in \Pi_{\phi_i}(U(m_i)) } \dim \pi_{i,f}^{K_i(\gn)} \right)^{n_i}.
\ees

\end{prop}

The first step in proving Proposition \ref{singlefinite} is to write both sides as a product over the finite places.  We describe this in the case of the first inequality, as the second is similar.  We have

\be
\label{finitefactor}
\sum_{\pi \in \Pi_\psi(G, \xi)} \dim \pi_f^{K(\gn)} = \prod_{v \nmid \infty} \sum_{\pi_v \in \Pi_{\psi_v}(G_v, \xi_v)} \dim \pi_v^{K_v(\gn)}
\ee
and

\bes
\prod_{i\ge 1} \left( \sum_{ \pi_i \in \Pi_{\phi_i}(U(m_i)) } \dim \pi_{i,f}^{K_i(\gn)} \right)^{n_i} = \prod_{v \nmid \infty} \prod_{i \ge 1}\left( \sum_{ \pi_{i,v} \in \Pi_{\phi_{i,v}}(U(m_i)) } \dim \pi_{i,v}^{K_{i,v}(\gn)} \right)^{n_i}.
\ees
It therefore suffices to prove that

\be
\label{factorineq}
\sum_{\pi_v \in \Pi_{\psi_v}(G_v, \xi_v)} \dim \pi_v^{K_v(\gn)} \le C_v (1 + O(q_v^{-2}) ) \prod_{i \ge 1} \left( \sum_{ \pi_{i,v} \in \Pi_{\phi_{i,v}}(U(m_i)) } \dim \pi_{i,v}^{K_{i,v}(\gn)} \right)^{n_i}
\ee
for all finite $v$.  Here, the constant $C_v$ may be arbitrary for $v \in S$, while for $v \notin S$ it is either 1 if $v$ is inert in $E/F$, or the $v$-component of the constant in (\ref{packetsum}) if $v$ is split.  It is important to keep $C_v$ independent of $\psi_v$ and $\gn$ (but it could depend on $K_v$) for the application to the proof of the main theorem. We divide the proof of \eqref{factorineq} into four cases, depending on whether $v$ is split in $E$, and whether $v \in S$. Thus the proof of Proposition \ref{singlefinite} will be complete by Lemmas \ref{sunramified}, \ref{sramified}, \ref{nsunramified}, and \ref{nsramified} below.


\subsection{$v$ split in $E/F$, $v \notin S$}

These $v$ are the only ones which require us to consider the special cases $m_1 = 1, 2, 3$ of Proposition \ref{singlefinite} separately.  In this case, the local packets under consideration each contain a single representation of $GL(N,F_v)$ or $GL(m_i,F_v)$.  The bound we prove, Lemma \ref{sunramified}, is an application of the fact that the representation in $\Pi_{\psi_v}(G_v, \xi_v)$ is a subquotient of an explicit induced representation.

\begin{lemma}
\label{sunramified}

Let $\Pi_{\psi_v}(G, \xi) = \{ \pi_v \}$ and $\Pi_{\phi_{i,v} }( U(m_i)) = \{ \pi_{i,v} \}$.  We have

\be
\label{sunramified1}
\dim \pi_v^{K(\gn)_v} \le (1 + 1/q_v)^{\sigma(\cS)} (1 + O(q_v^{-2}) ) N\gn_v^{\tau(\cS)} \prod_{i\ge 1} \left( \dim \pi_{i,v}^{K_{i}(\gn)_v} \right)^{n_i},
\ee
and if $m_1 = l$ with $l = 1, 2, 3$ we have

\be
\label{sunramified2}
\dim \pi_v^{K(\gn)_v} \le C(\epsilon, q_v) (1 + 1/q_v)^{\sigma_l(\cS)} (1 + O(q_v^{-2}) ) N\gn_v^{\tau_l(\cS)} \dim \pi_{1,v}^{K_1(\gn)_v} \prod_{i \ge 2} \left( \dim \pi_{i,v}^{K_i(\gn)_v} \right)^{n_i}.
\ee
The terms involving $1 + 1/q_v$ only need to be included if $v | \gn$.  The term $C(\epsilon, q_v) = 1$ if $l \neq 3$ or $q_v$ is greater than a constant depending on $\epsilon$.

\end{lemma}

\begin{proof}

We recall the identification $G_v = GL(N,E_w)$, which carries $K_v$ to $\widetilde{K}_w$.  View $\psi_v$ as a member of $\Psi(GL(N,E_w))$. As in Section \ref{notation}, we have an irreducible subquotient $\pi_v=\pi_{\psi_v}$ of an induced representation $\rho_w=\rho_{\psi_v}$ of $GL(N,E_w)$. Let $P_w$ denote the block upper triangular parabolic subgroup from which $\rho_w$ is induced. (So the Levi factor of $P_w$ is $\prod_{i=1}^k GL(m_i)^{n_i}$.)


We shall prove the first bound using $\dim \pi_v^{K(\gn)_v} \le \dim \rho_w^{\widetilde{K}(\gn)_w}$.  We have

\bes
\dim \rho_w^{\widetilde{K}(\gn)_w} = [\widetilde{K}_w : \widetilde{K}_w \cap \widetilde{K}(\gn)_w P_w] \prod_{i\ge 1} \left( \dim \mu_{i,w}^{\widetilde{K}_{i}(\gn)_w} \right)^{n_i}.
\ees
The result then follows from the fact that $[\widetilde{K}_w : \widetilde{K}_w \cap \widetilde{K}(\gn)_w P_w] = 1$ if $v \nmid \gn$, while if $v | \gn$ we have

\bes
[\widetilde{K}_w : \widetilde{K}_w \cap \widetilde{K}(\gn)_w P_w] =  (1 + 1/q_v)^{\sigma(\cS)} (1 + O(q_v^{-2}) ) N\gn_v^{\tau(\cS)},
\ees
and the fact that $\pi_{i,v}$ are isomorphic to $\mu_{i,w}$ so that $\dim \mu_{i,w}^{\widetilde{K}_{i}(\gn)_w} = \dim \pi_{i,v}^{K_{i}(\gn)_v}$.

In the case $m_1 = 1$, we define $P'_w$ to be the standard parabolic which is obtained by modifying $P_w$ in the upper-left $n_1 \times n_1$ block so that the $GL(1)^{n_1}$ factor in the Levi is replaced by $GL(n_1)$.  Let $\rho_w'$ be the representation induced from $P'_w$ using the same data as $\rho_w$, except that one takes the representation $\mu_{1,w} \circ \det$ on the new Levi factor $GL(n_1,E_w)$.  Because $\rho_w'$ is a quotient of $\rho_w$, we have $\dim \pi_v^{K(\gn)_v} \le \dim \rho_w'^{\widetilde{K}(\gn)_w}$, and

\bes
\dim \rho_w'^{\widetilde{K}(\gn)_w} = [\widetilde{K}_w : \widetilde{K}_w \cap \widetilde{K}(\gn)_w P'_w] \dim \mu_{1,w}^{\widetilde{K}_1(\gn)_w} \prod_{i \ge 2} \left( \dim \mu_{i,w}^{\widetilde{K}_{i}(\gn)_w} \right)^{n_i}.
\ees
The result follows as before, after calculating $[\widetilde{K}_w : \widetilde{K}_w \cap \widetilde{K}(\gn)_w P'_w]$.

In the cases $m_1 = 2, 3$, we bound all but one of the factors of $\dim \pi_{1,v}^{K_{1}(\gn)_v}$ in (\ref{sunramified1}) using the representation theory of $GL(m_1,F_v)$.  When $m_1 = 2$, we have
\bes
\dim \pi_{1,v}^{K_{1}(\gn)_v} \le (1 + 1/q_v) N\gn_v.
\ees
Indeed, the case where $\pi_{1,v}$ is an induced representation or a twist of Steinberg is immediate, while the supercuspidal case is Proposition \ref{GL2supercusp} (since $v$ is coprime to $2$ by our assumption that $S$ contains all places above 2).  Applying this in (\ref{sunramified1}) gives

\[
\dim \pi_v^{K(\gn)_v} \le (1 + 1/q_v)^{\sigma(\cS) + n_1 - 1} (1 + O(q_v^{-2}) ) N\gn_v^{\tau(\cS) + n_1 - 1} \dim \pi_{1,v}^{K_1(\gn)_v} \prod_{i\ge 2} \left( \dim \pi_{i,v}^{K_{i}(\gn)_v} \right)^{n_i},
\]
and as $\tau_2(\cS) = \tau(\cS) + n_1 - 1$ and $\sigma_2(\cS) = \sigma(\cS) + n_1 - 1$, this gives (\ref{sunramified2}) in this case.

When $m_1 = 3$, applying Corollary \ref{corinvdim} in (\ref{sunramified1}) gives

\[
\dim \pi_v^{K(\gn)_v} \le C(\epsilon, q_v) (1 + 1/q_v)^{\sigma(\cS)} (1 + O(q_v^{-2}) ) N\gn_v^{\tau(\cS) + 4(n_1 - 1) + \epsilon} \dim \pi_{1,v}^{K_1(\gn)_v} \prod_{i\ge 2} \left( \dim \pi_{i,v}^{K_{i}(\gn)_v} \right)^{n_i},
\]
which again gives (\ref{sunramified2}).

\end{proof}


\subsection{$v$ split in $E/F$, $v \in S$}

In this case, $G^*_v \simeq GL(N,F_v) \simeq GL(N,E_w)$ and $G_v$ is an inner form of $GL(N,F_v)$.  It is known (see \cite[Theorem 1.6.4]{KMSW} for instance) that the packet $\Pi_{\psi_v}(G^*)$ contains exactly one element whereas $\Pi_{\psi_v}(G, \xi)$ has one or zero elements.

For $v \in S$, the constant $C_v$ can be arbitrary.  This means that to prove (\ref{factorineq}), we need to know that the left hand side is bounded independently of $\psi_v$, and that if it is nonzero, then the right hand side is also nonzero.  Both facts are provided by the following local lemma, where we consider $\psi_v\in \Psi(G^*_v)$ and bounded $\phi_{i,v}\in \Phi(GL(m_i,F_v))$ with $\psi_v=\oplus_{i=1}^k \phi_{i,v}\boxtimes \nu(n_i)$. The unique representations in $\Pi_{\psi_v}(G, \xi)$ and $\Pi_{\phi_{i,v}}(U(m_i)_v)= \Pi_{\phi_{i,v}}(GL(m_i, F_v))$ are denoted by $\pi_v$ and $\pi_{i,v}$, respectively.

\begin{lemma}
\label{sramified}

There is $C(K_v) > 0$ such that $\dim \pi_v^{K_v} \le C(K_v)$.  For each $i$ there exists an open compact subgroup $K_{i,v} \subset U(m_i,F_v)$ depending only on $K_v$ such that the following is true for every $\psi_v$ and $\phi_{i,v}$ as above: if $\pi_v^{K_v} \neq 0$, then $\pi_{i,v}^{K_{i,v}} \neq 0$ for all $i$.

\end{lemma}

\begin{proof}

The first claim is Bernstein's uniform admissibility theorem \cite{Be}. (We need it just for unitary representations, but the proof there shows the theorem for irreducible admissible representations of general $p$-adic reductive groups.)

To prove the second claim, recall that $\psi_v$ gives rise to representations $\rho_{\psi_v}$ and $\pi_{\psi_v}$ of $G^*_v\simeq GL(N,F_v)$ as in Section \ref{notation}. So $\pi_{\psi_v}$ is an irreducible subquotient of $\rho_{\psi_v}$.

The hypothesis $\pi_v^{K_v} \neq 0$ means that $1_{K_v}(\pi_v) \neq 0$.  If we transfer $1_{K_v}$ to a function $1_{K_v}^*$ on $G^*_v$, we have the character identity $1_{K_v}^*(\pi_{\psi_v}) =e(G_v)a_{\psi_v} 1_{K_v}(\pi_v) \neq 0$ by Theorem 1.6.4 (1) of \cite{KMSW} with certain signs $e(G_v),a_{\psi_v}\in \{\pm1\}$.  If we let $K'_v \subset GL(N,F_v)$ be an open compact subgroup such that $1_{K_v}^*$ is bi-invariant under $K'_v$, this implies that $\pi_{\psi_v}^{K'_v} \neq 0$ and thus $\rho_{\psi_v}^{K'_v}\neq 0$.  This gives $\pi_{i,v}^{K_{i,v}} \neq 0$ for suitable $K_{i,v} \subset GL(m_i,F_v)$, which implies the claim. (To see this, one uses a description of invariant vectors in an induced representation under an open compact subgroup as in the first display of \cite[p.26]{B}, noting that the double coset $P\backslash G/K$ there is finite.)



\end{proof}

\subsection{$v$ nonsplit in $E/F$, $v \notin S$}

In this case, for each $\psi_v\in \Psi(G^*_v)$ we have $\psi^N_v=\eta_{\chi_{\kappa}}\circ \psi_v\in \Psi(\tilde G(N)) = \Psi(GL(N,E_w))$. This gives rise to representations $\pi_{\psi^N_v}$ and $\rho_{\psi^N_v}$ of $GL(N,E_w)$ as in Section \ref{notation}. Similarly $\phi_{i,v}\in \Phi(U(m_i)_v)$ gives a representation $\pi_{\phi_{i,v}^{m_i}}$ of $GL(m_i,E_w)$ for the parameter $\eta_{\chi_{\kappa_i}}\circ \phi_{i,v}$. If $\psi_v$ and $\phi_{i,v}$ arise from global data as at the start of Section \ref{singleparameter} then $\pi_{\phi_{i,v}^{m_i}}$ is nothing but $\mu_{i,w}$.

Inequality (\ref{factorineq}) in this case follows from the lemma below.

\begin{lemma} Consider $\psi_v\in \Psi(G^*_v)$ and $\phi_{i,v}\in \Phi(U(m_i)_v)$ as above such that $\psi^N_v = \oplus_{i\ge 1} \phi_{i,v}^{m_i} \boxtimes \nu(n_i)$. Then
\label{nsunramified}
we have

\be
\label{dimle1}
\sum_{\pi_v \in \Pi_{\psi_v}(G,\xi)} \dim \pi_v^{K_v} \le 1.
\ee
If equality holds, then

\be
\label{dimequality}
\sum_{\pi_{i,v} \in \Pi_{\phi_{i,v}}(U(m_i))} \dim \pi_{i,v}^{K_{i,v}} = 1
\ee
for all $i$.

\end{lemma}

\begin{proof}

Suppose first that $s_{\psi_v} \in \{\pm 1\}$. We have a hyperspecial subgroup $\widetilde{K}_v$ of $G(N)_v \simeq GL(N,E_w)$. The twisted fundamental lemma implies that the functions $1_{K_v}$ and $1_{\widetilde{K}_v \rtimes \theta}$ are related by transfer.

Applying the character identity for $U(N)$ (Theorem 3.2.1 (b) of \cite{M}) with $s = 1$ gives

\bes
1_{K_v}^{U(N)}(\psi_v) = \sum_{\pi_v \in \Pi_{\psi_v}(G,\xi)} \dim \pi_v^{K_v},
\ees
and combining this with the twisted character identity \cite[Theorem 3.2.1 (a)]{M} and the twisted fundamental lemma gives

\bes
\sum_{\pi_v \in \Pi_{\psi_v}(G,\xi)} \dim \pi_v^{K_v} = \tr( \widetilde{\pi}_{\psi^N_v}( 1_{\widetilde{K}_v \rtimes \theta}) ).
\ees
The twisted trace $\tr( \widetilde{\pi}_{\psi^N_v}( 1_{\widetilde{K}_v \rtimes \theta}) )$ is equal to the trace of $\widetilde{\pi}_{\psi^N_v}(\theta)$ on $\pi_{\psi^N_v}^{\widetilde{K}_v}$, so we have

\bes
\tr( \widetilde{\pi}_{\psi^N_v}( 1_{\widetilde{K}_v \rtimes \theta}) ) \le \dim \pi_{\psi^N_v}^{\widetilde{K}_v}.
\ees
Since $\pi_{\psi^N_v}$ is a subquotient of $\rho_{\psi^N_v}$, we have

\bes
\dim \pi_{\psi^N_v}^{\widetilde{K}_v} \le \dim \rho_{\psi^N_v}^{\widetilde{K}_v} \le 1
\ees
which gives (\ref{dimle1}).

If equality holds, then $\psi^N_v$ is unramified. So all $\phi_{i,v}$ are unramified as well. Applying  \cite[Theorem 3.2.1 (b)]{M} to the parameter $\phi_{i,v}$ and the function $1_{K_{i,v}}$ for $U(m_i)$ gives

\bes
\sum_{\pi_{i,v} \in \Pi_{\phi_{i,v}}(U(m_i))} \dim \pi_{i,v}^{K_{i,v}} = 1_{K_{i,v}}^{U(m_i)}(\phi_{i,v}).
\ees
If $\widetilde \pi_{\phi_{i,v}^{m_i}}$ is the canonical extension of $\pi_{\phi_{i,v}^{m_i}}$ to $\widetilde{G}(m_i)_v$ (via Whittaker normalization),

\bes
\sum_{\pi_{i,v} \in \Pi_{\phi_{i,v}}(U(m_i))} \dim \pi_{i,v}^{K_{i,v}} = \tr( \widetilde \pi_{\phi_{i,v}^{m_i}}( 1_{\widetilde{K}_{i,v} \rtimes \theta}) ).
\ees
$\tr( \widetilde \pi_{\phi_{i,v}^{m_i}}( 1_{\widetilde{K}_{i,v} \rtimes \theta}) )$ is the trace of $\theta$ on the one-dimensional space $\pi_{\phi_{i,v}^{m_i}}^{\widetilde{K}_{i,v} }$, so we have $\tr( \widetilde \pi_{\phi_{i,v}^{m_i}}( 1_{\widetilde{K}_{i,v} \rtimes \theta}) ) = \pm 1$, and (\ref{dimequality}) follows from positivity.

Now suppose that $s_{\psi_v} \notin \{\pm 1\}$, and let $(G^\e, s^\e, \eta^\e)$ be the elliptic endoscopic triple for $G$ with $s^\e = s_{\psi_v}$.  We have $G^\e = U(a) \times U(b)$ for some $a, b > 0$.  There is an Arthur parameter $\psi^\e$ for $G^\e$ such that $\eta^\e \circ \psi^\e = \psi$, which we may factorise as $\psi^\e = \psi_1 \times \psi_2$.  We let $K^\e_v \subset G^\e(F_v)$ be a hyperspecial subgroup, and let $1^\e_{K_v}$ be the characteristic function of $K^\e_v$.  The Fundamental Lemma implies that $1_{K_v} \in \cH(G_v)$ and $1^\e_{K_v} \in \cH(G^\e_v)$ have $\Delta[\e, \xi, z]$-matching orbital integrals.  Applying \cite[Theorem 3.2.1 (b)]{M} with $s = s_{\psi_v}$ gives

\bes
\sum_{\pi_v \in \Pi_{\psi_v}(G,\xi)} 1_{K_v}(\pi_v) = 1^\e_{K_v}(\psi_v^\e) = 1_{K_{1,v}}(\psi_{1,v}) 1_{K_{2,v}}(\psi_{2,v}).
\ees
The result now follows by applying the result in the case $s_{\psi_v} \in \{\pm 1\}$ to the groups $U(a)$ and $U(b)$.

\end{proof}

\subsection{$v$ nonsplit in $E/F$, $v \in S$} Here we prove a result analogous to Lemma \ref{sramified}.

\begin{lemma}
\label{nsramified}

There exist open compact subgroups $K_{i,v} \subset U(m_i)_v$ depending only on $K_v$ such that the following holds: given $\psi_v\in \Psi(G_v)$ and $\phi_{i,v}\in \Phi(U(m_i)_v)$ such that $\psi^N_v=\oplus_{i=1}^k \phi^N_{i,v}\boxtimes \nu(n_i)$ (thus $\phi_{i,v}$ are bounded), if

\bes
\sum_{\pi_v \in \Pi_{\psi_v}(G,\xi)} \dim \pi_v^{K_v} \neq 0 \quad \text{then} \quad \sum_{\pi_{i,v} \in \Pi_{\phi_{i,v}}(U(m_i)) } \dim \pi_{i,v}^{K_{i,v}} \neq 0.
\ees
Moreover there is a constant $C(K_v) > 0$ which is independent of $\psi_v$ such that
\bes
\sum_{\pi_v \in \Pi_{\psi_v}(G,\xi)} \dim \pi_v^{K_v} \le C(K_v).
\ees

\end{lemma}

\begin{proof}

We begin with the first claim.  Suppose $s_{\psi_v} \in \{\pm 1\}$.  Let $1^*_{K_v}$ be the transfer of $1_{K_v}$ to $G^*_v$.  The character identity of \cite[Thm 1.6.1 (4)]{KMSW} gives

\bes
0 \neq e(G_v) \sum_{\pi_v \in \Pi_{\psi_v}(G,\xi)} \dim \pi_v^{K_v} = 1^*_{K_v}(\psi_v),
\ees
where $e(G_v) \in \{ \pm 1 \}$; note that the coefficients $\langle \pi, 1 \rangle$ appearing in the cited theorem are all 1 (where we take $s^{\mathfrak e}=s_{\psi_v}$).  Using the surjectivity result of Mok \cite[Prop. 3.1.1 (b)]{M}, there is a function $\widetilde{1}_{K_v}$ on $\widetilde{G}(N)_v$ whose twisted transfer to $G_v^*$ is $1^*_{K_v}$, and so we have $1^*_{K_v}(\psi_v) = \tr( \widetilde{\pi}_{\psi_v}( \widetilde{1}_{K_v}) )$.  Let $\widetilde{K}_v \subset G(N)_v$ be a compact open subgroup such that $\widetilde{1}_{K_v}$ is bi-invariant under $\widetilde{K}_v$.  It follows that we must have $\pi_{\psi_v}^{ \widetilde{K}_v } \neq 0$, and hence there are compact open $\widetilde{K}_{i,v} \subset G(m_i)_v$ depending only on $K_v$ such that $\pi_{\phi_{i,v}}^{ \widetilde{K}_{i,v} } \neq 0$.  The result now follows from Lemma \ref{cpctdescent} below.

Now suppose that $s_{\psi_v} \notin \{\pm 1\}$, and let $(G^\e, s^\e, \eta^\e)$ be the elliptic endoscopic triple for $G$ with $s^\e = s_{\psi_v}$ and so $G^\e = U(a) \times U(b)$ for some $a, b > 0$.  There is an Arthur parameter $\psi^\e$ for $G^\e$ such that $\eta^\e \circ \psi^\e = \psi$, which we may factorise as $\psi^\e = \psi_1 \times \psi_2$.  Let $1_{K_v}^\e$ be the function obtained by transferring $1_{K_v}$ to $G^\e_v$.  Applying the trace identity

\bes
e(G_v) \sum_{\pi_v \in \Pi_{\psi_v}(G,\xi)} \dim \pi_v^{K_v} = 1^\e_{K_v}(\psi_v^\e)
\ees
gives $1^\e_{K_v}(\psi_v^\e) \neq 0$.  Because $1^\e_{K_v}(\psi_v^\e)$ is equal to a sum of traces there is a compact open $K_{1,v} \times K_{2,v} \subset G^\e_v$ such that

\bes
1_{K_{1,v} \times K_{2,v}}(\psi^\e_v) = 1_{K_{1,v}}(\psi_{1,v}) 1_{K_{2,v}}(\psi_{2,v}) \neq 0
\ees
and the result now follows from the case $s_{\psi_v} \in \{\pm 1\}$ for the groups $U(a)$ and $U(b)$.

We now prove the second claim. Suppose $s_{\psi_v} \in \{\pm 1\}$. We again use the identity
$$e(G_v) \sum_{\pi_v \in \Pi_{\psi_v}(G,\xi)} \dim \pi_v^{K_v} = \tr(\widetilde{\pi}_{\psi_v}(\widetilde{1}_{K_v}) ),$$
 and let $\widetilde{K}_v \subset G(N)_v$ be a compact open subgroup such that $\widetilde{1}_{K_v}$ is bi-invariant under $\widetilde{K}_v$.  The trace $\tr(\widetilde{\pi}_{\psi_v}(\widetilde{1}_{K_v}) )$ is equal to the trace of $\widetilde{\pi}_{\psi_v}(\widetilde{1}_{K_v})$ on the space $\pi_{\psi_v}^{ \widetilde{K}_v}$, and the operator norm of $\widetilde{\pi}_{\psi_v}(\widetilde{1}_{K_v})$ is at most $\| \widetilde{1}_{K_v} \|_1 = C(K_v)$.  We therefore have $\left| \tr(\widetilde{\pi}_{\psi_v}(\widetilde{1}_{K_v}) ) \right|\le C(K_v) \dim  \pi_{\psi_v}^{\widetilde{K}_v}$, and the result follows as in Lemma \ref{sramified}.  If $s_{\psi_v} \notin \{\pm 1\}$, we reduce to the case of $U(a) \times U(b)$ as before.

\end{proof}

  Recall that $\eta_{\chi_{\kappa_i}}\circ \phi_{i,v}\in \Phi(G(m_i)_v)$ corresponds to $\mu_{i,w}$ via local Langlands under the isomorphism $G(m_i)_v \simeq GL(m_i,E_w)$, where $w$ is the unique place of $E$ above $v$.

\begin{lemma}
\label{cpctdescent}

If $\widetilde{K}_{i,w} \subset GL(m_i, E_w)$ is a compact open subgroup, then there is a compact open subgroup $K_{i,v} \subset U(m_i)_v$ with the following property: For any bounded parameter $\phi_{i,v}\in \Phi(U(m_i)_v)$ and the representation $\mu_{i,w}$ of $GL(m_i,E_w)$ corresponding as above, if $\mu_{i,w}^{\widetilde{K}_{i,w}} \neq 0$ then

\be\label{cpctdescenteq}
\sum_{\pi_{i,v} \in \Pi_{\phi_{i,v}}(U(m_i)) } \dim \pi_{i,v}^{K_{i,v}} \neq 0.
\ee

\end{lemma}

\begin{proof}

The only nontrivial part of the lemma is the assertion that $K_{i,v}$ may be chosen independently of $\mu_{i,w}$. To this end, we will show that $\phi_{i,v}$ (or $\mu_{i,w}$) varies over a compact domain and that $K_{i,v}$ as in the lemma can be chosen in open neighborhoods. Then the proof will be complete by taking intersection of the finitely many $K_{i,v}$ for a finite open covering.

 By a theorem of Jacquet, our assumption that $\mu_{i,w}$ was tempered implies that $\mu_{i,w}$ belongs to a family of full induced representations from some
\be\label{inducingrep}
\mu_1' | \cdot |^{is_1} \otimes \cdots \otimes \mu_k' | \cdot |^{is_k}
\ee
with $\mu'_j$ square integrable and $s_j \in \R / (2\pi / \log q_w)\Z$.  Our assumption that $\mu_{i,w}$ had bounded depth implies that the set of tuples $\mu_1', \ldots, \mu_k'$ we must consider is finite, and so we only need to consider one.  We then need to show that the set of $s_j$ such that $\mu_{i,w}$ is conjugate self-dual (i.e. $\mu_{i,w}\simeq \mu_{i,w}\circ \theta$) is compact.  Because $\mu_j' | \cdot |^{is_j} \circ \theta = (\mu_j' \circ \theta) | \cdot |^{-is_j}$, this condition is equivalent to saying that the multisets $\{ \mu_j' | \cdot |^{is_j} \}$ and $\{ (\mu_j' \circ \theta) | \cdot |^{-is_j} \}$ are equivalent.  This in turn is equivalent to the existence of a permutation $\sigma \in S_k$ such that $\mu_j' | \cdot |^{is_j} \simeq (\mu_{\sigma(j)}' \circ \theta) | \cdot |^{-is_{\sigma(j)} }$ for all $j$.  For each $\sigma$ the set of $s_j$ satisfying this is closed, and hence the set of $s_j$ such that $\mu_{i,w}$ is conjugate self-dual is closed and compact.



  For fixed $s_1, \ldots, s_k$, there is some $K_{i,v}$ such that $1_{K_{i,v}}(\phi_{i,v}) \neq 0$, where $1_{K_{i,v}}(\phi_{i,v})$ is equal to the left hand side of \eqref{cpctdescenteq} by definition.  Also, if we transfer $1_{K_{i,v}}$ to $\widetilde{1}_{K_{i,v}}$ on $\widetilde{G}(m_i)$ using the surjectivity theorem of Mok \cite[Prop. 3.1.1 (b)]{M} then the character identity tells us that $1_{K_{i,v}}(\phi_{i,v})  = \tr( \widetilde{\pi}_{\phi_{i,v}}^{m_i}( \widetilde{1}_{K_{i,v}}) )$ (where the twisted trace is Whittaker normalized). The point is that $\tr( \widetilde{\pi}_{\phi_{i,v}}^{m_i}( \widetilde{1}_{K_{i,v}}) )$ varies continuously in the $s_j$ (see \cite{Ro}) so we still have $1_{K_{i,v}}(\phi_{i,v}) \neq 0$ around an open neighborhood of $s_1,...,s_k$ (where $\phi_{i,v}$ varies as $s_1,...,s_k$ vary). The result now follows by compactness.

\end{proof}



\section{Archimedean control on parameters}
\label{Archcontrol}


In this section, we prove some useful conditions on the parameters $\psi$ that contribute to the cohomology of $X(\gn)$.

Given $\phi_\infty=\otimes_{v|\infty} \phi_v \in \Phi(U(n)_\infty)$ for $n\ge 1$, note that the restriction of $\phi_v$ to $W_\C=\C^\times$ (for a fixed isomorphism $\overline F_v\simeq \C$), viewed as an $n$-dimensional representation via $\widehat{U(n)}=GL(n,\C)$, is a direct sum of $n$ characters $z\mapsto z^{a_{i,v}}\overline{z}^{b_{i,v}}$ with $a_{i,v},b_{i,v}\in \C$ and $a_{i,v}-b_{i,v}\in \Z$ for $i=1,...,n$. We say that $\phi_\infty$ is C-algebraic if $n$ is odd and all $a_{i,v}\in \Z$ or if $n$ is even and all $a_{i,v}\in \frac12+\Z$. We say $\phi_v$ is regular if $a_{i,v}$ are distinct. If $\pi_\infty$ is a member of the $L$-packet for $\phi_\infty$ then $\pi_\infty$ is said to be regular or C-algebraic if $\phi_\infty$ is. (This is Clozel's definition and coincides with the general definition \cite[Definition 2.3.3]{BG} for general reductive groups.)

Let $\cS = (n_1, m_1), \ldots, (n_k, m_k)$ be a shape as in Section \ref{singleparameter}. If $\psi  \in \Psi_2(G^*, \eta_{\chi_\kappa} )_\cS$ then $\psi$ gives rise to $\mu_i\in  \widetilde\Phi_{\mathrm{sim}}(m_i)$ and $\phi_i \in \Phi_{\mathrm{sim}}(U(m_i),\eta_{\chi_{\kappa_i}})$ as before.

\begin{lemma}\label{Ramanujan}
Let $\psi \in \Psi_2(G^*, \eta_{\chi_\kappa} )_\cS$.  If there is $\pi_\infty \in \Pi_{\psi_\infty}(G, \xi)$ with $h^d(\g, K_\infty; \pi_\infty) \neq 0$ then $\phi_{i,\infty}$ or $\phi_{i,\infty}\otimes \chi_{-,\infty}$ is C-algebraic. Moreover every $\phi_{i,v}$ is bounded at every place $v$ (equivalently $\mu_{i,v}$ is tempered at every place $v$).
\end{lemma}

\begin{proof}
   Since $\pi_\infty$ contributes to cohomology, its infinitesimal character is equal to that of the trivial representation. In particular it is regular C-algebraic, cf. \cite[Lemmas 7.2.2, 7.2.3]{BG}. Hence $\phi_{\psi_\infty}$ is regular C-algebraic. (Here we use the simple recipe to determine the infinitesimal character of $\pi_\infty$ from $\phi_{\psi_\infty}$ by differentiation, as described in \cite[Section 2.1]{MR}.) For each infinite place $v$, the representation $\phi_{i,v}|_{W_\C}$ is the direct sum of $m_i$ characters, say $\eta_{i,1}$, ..., $\eta_{i,m_i}$. Then $\oplus_{j=1}^{m_i} \oplus_{l=0}^{n_i-1} \eta_{i,j}|\cdot|^{\frac{n_i-1}{2}-l}$ appears as a subrepresentation of $\phi_{\psi_\infty}$. As the latter is regular C-algebraic, we see that for each $v|\infty$, $\phi_{i,v}|_{W_\C}$ is regular and that either $\phi_{i,v}|_{W_\C}$ or $\phi_{i,v}|_{W_\C}\otimes |\cdot|^{1/2}$ is C-algebraic, depending on the parity of $N-n_i$. It follows that $\phi_{i,\infty}|_{W_\C}$ or $\phi_{i,\infty}|_{W_\C}\otimes |\cdot|^{1/2}$ is regular C-algebraic.
    By the definition of $\chi_-$ in Section \ref{notation}, $\phi_{i,\infty}|_{W_\C}\otimes |\cdot|^{1/2}$ is C-algebraic if and only if $\phi_{i,\infty}|_{W_\C}\otimes \chi_{-,\infty}$ is.

 The key point is that $\mu_i$ or $\mu_i \otimes \chi_{-}$ is an automorphic representation with regular C-algebraic component at $\infty$ (recalling that $\mu_{i,\infty}$ lies in the packet for $\phi_{i,\infty}$). Both $\mu_i$ and $\mu_i \otimes \chi_{-}$ are cuspidal and conjugate self-dual, so either $\mu_i$ or $\mu_i \otimes \chi_{-}$ (whichever is C-algebraic at $\infty$) is essentially tempered at all finite places by \cite[Theorem 1.2]{Ca} (the cohomological condition in \emph{loc. cit.} follows from regular C-algebraicity, cf. \cite[Lemme 3.14]{Cl}) and at all infinite places by \cite[Lemme 4.9]{Cl}. In either case, twisting by $\chi_-$ if necessary, we deduce that $\mu_i$ is essentially tempered everywhere. Since the central character of $\mu_i$ is unitary, we see that $\mu_i$ is tempered everywhere. By the local Langlands correspondence \cite[Theorem 2.5.1 (b)]{M}, this is equivalent to $\phi_{i,v}$ being bounded at every $v$.
\end{proof}

\begin{lemma}
\label{RACSD}

For each $i$, there is a finite set of parameters $\mathcal P_{i,\infty} \subset {\Phi}(U(m_i)_\infty)$ with the following property:  If $\psi  \in \Psi_2(G^*, \eta_{\chi_\kappa} )_\cS$, and there exists $\pi_\infty \in \Pi_{\psi_\infty}(G, \xi)$ with $h^d(\g, K_\infty; \pi_\infty) \neq 0$, then $\phi_{i,\infty} \in \mathcal P_{i,\infty}$.

\end{lemma}

\begin{proof}

The infinitesimal character of $\pi_\infty$ is determined by the condition that $h^d(\g, K_\infty; \pi_\infty) \neq 0$ (to be the half sum of all positive roots of $G$), thus there are finitely many such $\pi_\infty$. So they are contained in finitely many Arthur packets, whose parameters form a finite subset $\mathcal P\subset \Psi(U(N)_\infty)$. If $\psi$ gives rise to $\phi_i$ then $\eta_{\chi_{\kappa_i}}\circ(\phi_{i,v}\boxtimes \nu(n_i))$ should appear as a factor of $\eta_{\chi_\kappa} \circ \psi_v$ for every infinite place $v$. Since we have the constraint $\otimes_{v|\infty}\psi_v\in \mathcal P$, it is clear that there are finitely many possibilities for $\phi_{i,\infty}$.

%
%

\end{proof}

\section{Summing over parameters}\label{summing}

In this section we continue the proof of Theorem \ref{maintheorem} from the end of Section \ref{application} and finish the proof.
In the preliminary bound \eqref{firststep}, we will fix a shape $\cS$ and bound the contribution to $h^d_{(2)}(X(\gn))$ from parameters in $\Psi_2(G^*, \eta_{\chi_\kappa})_\cS$, which we denote by $h^d_{(2)}(X(\gn))_\cS$. Clearly it suffices to establish a bound for $h^d_{(2)}(X(\gn))_\cS$ as in Theorem \ref{maintheorem}.

 Suppose $\psi \in \Psi_2(G^*, \eta_{\chi_\kappa})_\cS$ has the property that there is $\pi \in \Pi_\psi(G, \xi)$ with $h^d(\g, K_\infty; \pi_\infty) \neq 0$.  Proposition \ref{cohreps} implies that $\pi_{v_0}$ must be a Langlands quotient of a standard representation with an exponent of the form $(z/\overline{z})^{p/2} (z \overline{z})^{(a-1)/2}$ for some $a \ge N-d$.  Proposition 13.2 of \cite{BMM} implies that there is $i$ such that $n_i \ge N-d$, and we assume that this is $n_1$.  Note that \cite[Prop 13.2]{BMM} implicitly assumes that the other archimedean components of $\pi$ have regular infinitesimal character, but this is satisfied in our case.

Apply Lemma \ref{RACSD} to obtain finite sets $\mathcal P_{i,\infty} \subset \Phi(U(m_i)_\infty)$ for all $i$ such that if $\psi \in \Psi_2(G^*,  \eta_{\chi_\kappa})_\cS$, and there is $\pi \in \Pi_\psi(G,\xi)$ with $h^d(\g, K_\infty; \pi_\infty) \neq 0$, then $\phi_{i,\infty} \in \mathcal P_{i,\infty}$.  Let $\Psi_\text{rel}$ be the set of $A$-parameters $\psi \in \Psi_2(G^*, \eta_{\chi_\kappa})_\cS$ with $\phi_i$ bounded everywhere and $\phi_{i,\infty} \in \mathcal P_{i,\infty}$ for all $i$. By Lemmas \ref{Ramanujan} and \ref{RACSD} we have

\bes
h^d_{(2)}(X(\gn))_\cS \le \sum_{\psi \in \Psi_\text{rel} } \sum_{\pi \in \Pi_\psi(G, \xi) } h^d(\g, K_\infty; \pi_\infty) \dim \pi_f^{K(\gn)},
\ees
and because $h^d(\g, K_\infty; \pi_\infty)$ is bounded we may simplify this to

\bes
h^d_{(2)}(X(\gn))_\cS \ll \sum_{\psi \in \Psi_\text{rel} } \sum_{\pi \in \Pi_\psi(G, \xi) } \dim \pi_f^{K(\gn)}.
\ees
Because $\phi_i$ is bounded everywhere for every $i$, we may apply Proposition \ref{singlefinite} to obtain

\bes
h^d_{(2)}(X(\gn))_\cS \ll \prod_{v | \gn} (1 + 1/q_v)^{\sigma(\cS)} N\gn^{\tau(\cS)} \sum_{\psi \in \Psi_\text{rel} } \prod_{i \ge 1}  \left( \sum_{ \pi_i \in \Pi_{\phi_i}(U(m_i)) } \dim \pi_{i,f}^{K_i(\gn)} \right)^{n_i}.
\ees
Let $\Phi_\text{sim}^\text{bdd}(U(m_i),\eta_{\chi_{\kappa_i}})$ denote the set of simple parameters that are bounded everywhere. Taking a sum over $\psi \in \Psi_\text{rel}$ corresponds to taking a sum over the possibilities for $\phi_i \in \Phi_\text{sim}^\text{bdd}(U(m_i),\eta_{\chi_{\kappa_i}})$ with $\phi_{i,\infty} \in \mathcal P_{i,\infty}$.  We may therefore factorize the sum over $\psi$ to ones over $\phi_i$, which gives

\begin{align}
\notag
h^d_{(2)}(X(\gn))_\cS & \ll \prod_{v | \gn} (1 + 1/q_v)^{\sigma(\cS)} N\gn^{\tau(\cS)} \prod_{i \ge 1} \sum_{ \substack{ \phi_i \in \Phi_\text{sim}^\text{bdd}(U(m_i),\eta_{\chi_{\kappa_i}}) \\ \phi_{i,\infty} \in \mathcal P_{i,\infty} } } \left(  \sum_{ \pi_i \in \Pi_{\phi_i}(U(m_i)) } \dim \pi_{i,f}^{K_i(\gn)} \right)^{n_i} \\
\label{factor2}
& \le \prod_{v | \gn} (1 + 1/q_v)^{\sigma(\cS)} N\gn^{\tau(\cS)} \prod_{i \ge 1}  \left( \sum_{ \substack{ \phi_i \in \Phi_\text{sim}^\text{bdd}(U(m_i),\eta_{\chi_{\kappa_i}}) \\ \phi_{i,\infty} \in \mathcal P_{i,\infty} } } \sum_{ \pi_i \in \Pi_{\phi_i}(U(m_i)) } \dim \pi_{i,f}^{K_i(\gn)} \right)^{n_i}.
\end{align}

We may bound the sums using the global limit multiplicity formula of Savin \cite{Sa}.  Indeed, because $\phi_i$ is a simple generic parameter, the packet $\Pi_{\phi_i}(U(m_i))$ is stable, so that every representation $\pi_i \in \Pi_{\phi_i}(U(m_i))$ occurs in the discrete spectrum of $U(m_i)$ with multiplicity one.  In fact, $\pi_i$ must actually lie in the cuspidal spectrum by \cite[Theorem 4.3]{Wa}, because $\pi_{i,\infty}$ is tempered.  Because the archimedean components of $\phi_i$ are restricted to finite sets, there is a finite set $\Pi_{i, \infty}$ of representations of $U(m_i)_\infty$ such that if $\pi_i \in \Pi_{\phi_i}(U(m_i))$ then $\pi_{i,\infty} \in \Pi_{i,\infty}$.  If $m_\text{cusp}(\pi_\infty, Y_i(\gn) )$ denotes the multiplicity with which an irreducible representation $\pi_\infty$ of $U(m_i)_\infty$ occurs in the $L^2$-space of cuspforms $L^2_\text{cusp}(Y_i(\gn))$, where $Y_i(\gn) = U(m_i,F) \backslash U(m_i,\A) / K_i(\gn)$, we have

\bes
\sum_{ \substack{ \phi_i \in \Phi_\text{sim}^\text{bdd}(U(m_i),\eta_{\chi_{\kappa_i}}) \\ \phi_{i,\infty} \in \mathcal P_{i,\infty} } } \sum_{ \pi_i \in \Pi_{\phi_i}(U(m_i)) } \dim \pi_{i,f}^{K_i(\gn)} \le \sum_{ \substack{ \pi_i \subset L^2_\text{cusp}([U(m_i)]) \\ \pi_{i,\infty} \in \Pi_{i,\infty} } } \dim \pi_{i,f}^{K_i(\gn)} = \sum_{\pi_\infty \in \Pi_{i,\infty} } m_\text{cusp}( \pi_\infty, Y_i(\gn) ).
\ees
For each $\pi_\infty$, Savin \cite{Sa} gives

\bes
m_\text{cusp}( \pi_\infty, Y_i(\gn) ) \ll [K_i : K_i(\gn) ] \ll \prod_{v|\gn} (1 - 1/q_v) N\gn^{m_i^2},
\ees
and combining this with (\ref{factor2}) gives

\bes
h^d_{(2)}(X(\gn))_\cS \ll \prod_{v|\gn} (1-1/q_v) N\gn^{\tau'(\cS)},
\ees
where $\tau'(\cS) = \tau(\cS) + \sum_{i \ge 1} n_i m_i^2$.  If $1\le m_1 \le 3$ then applying Proposition \ref{singlefinite} and working as above gives

\bes
h^d_{(2)}(X(\gn))_\cS \ll \prod_{v|\gn} (1+1/q_v)^{\sigma'_l(\cS)} N\gn^{\tau'_l(\cS)},
\ees
where $l = m_1$, $\tau'_l(\cS) = \tau_l(\cS) + m_1^2 + \sum_{i \ge 2} n_i m_i^2$, and $\sigma'_l(\cS) = \sigma_l(\cS) - 1 - \sum_{i \ge 2} n_i$.

The bounds for the functions $\tau'$ and $\tau'_j$ given by Lemma \ref{comb} below then imply that
\bes
h^d_{(2)}(X(\gn))_\cS \ll_\epsilon N\gn^{Nd + \epsilon},
\ees
unless we are in one of the two cases listed there.  In the exceptional case (\ref{equal2}) we have $\sigma'_2(\cS) = 1$ and hence $h^d_{(2)}(X(\gn))_\cS \ll \prod_{v|\gn} (1+1/q_v) N\gn^{Nd + 1}$, and in case (\ref{equal1}), which should give the general main term, we have

\bes
h^d_{(2)}(X(\gn))_\cS \ll \prod_{v|\gn} (1-1/q_v) N\gn^{Nd+1}.
\ees
This completes the proof of Theorem \ref{maintheorem}. $\Box$

\smallskip

  It remains to prove the lemma used in the above proof.

\begin{lemma}
\label{comb}

If $m_1 \ge 4$, we have $\tau'(\cS) \le Nd$.  If $m_1 = l$, $l = 1, 2, 3$, we have $\tau'_l(\cS) \le Nd + \delta_{3l} \epsilon$, except in the following cases where $\tau'_l(\cS) = Nd + 1$.

\begin{enumerate}

\item
\label{equal1} $\cS = (N-d,1), (1,d)$.

\item
\label{equal2} $\cS = (2,2)$ and $d = 2$.

\end{enumerate}

\end{lemma}

\begin{proof}

We begin with the case $m_1 \ge 4$.  The inequality $d \ge N - n_1$ implies that it suffices to prove $\tau'(\cS) \le N(N - n_1)$.  Substituting the definition of $\tau'$ and simplifying, we must show that

\be
\label{eliminated}
\binom{N}{2} + \sum_{i \ge 1} n_i \binom{m_i + 1}{2} \le N(N - n_1).
\ee
We next eliminate the variables other than $N$, $m_1$ and $n_1$.  The identity $\binom{n+1}{2} = 1 + \ldots + n$ implies that if $A = \sum a_i$, then $\binom{A+1}{2} \ge \sum \binom{a_i + 1}{2}$, and applying this to the $m_i$ with multiplicity $n_i$ for $i \ge 2$ gives

\be
\label{binom}
\sum_{i \ge 2} n_i \binom{m_i+1}{2} \le \binom{N - n_1 m_1 + 1}{2}.
\ee
Note that equality occurs above if and only if $\sum_{i \ge 2} n_i$ is either 0 or 1.  After applying this in (\ref{eliminated}), we are reduced to showing that

\bes
\binom{N}{2} + n_1 \binom{m_1+1}{2} + \binom{N - n_1 m_1 + 1}{2} \le N(N - n_1).
\ees
Simplifying gives

\begin{align*}
N(N-1) + n_1 (m_1+1) m_1 + (N - n_1 m_1 + 1)(N - n_1 m_1) & \le 2N(N-n_1)  \\
-2m_1 n_1 N + 2N n_1 + m_1^2 n_1^2 & \le - m_1^2 n_1  \\
0 & \le 2m_1 N - 2N - m_1^2 n_1 - m_1^2
\end{align*}
As $N \ge m_1 n_1$, we have $m_1 N \ge m_1^2 n_1$ so that

\bes
2m_1 N - 2N - m_1^2 n_1 - m_1^2 \ge (m_1 - 2)N - m_1^2.
\ees
Because $n_1 \ge 2$ we have $N \ge n_1m_1 \ge 2m_1$, so that

\bes
(m_1 - 2)N - m_1^2 \ge m_1^2 - 4m_1 \ge 0,
\ees
where we have used $m_1 \ge 4$ at the last step.

In the case $m_1 = 1$, we have

\bes
\tau'_1(\cS) = \binom{N}{2} - \binom{n_1}{2} + \sum_{i \ge 2} n_i \binom{m_i +1}{2} + 1,
\ees
and applying (\ref{binom}) gives

\bes
\tau'_1(\cS) \le \binom{N}{2} - \binom{n_1}{2} + \binom{N - n_1 + 1}{2} + 1.
\ees
It may be seen that the right hand side of this simplifies to $N(N - n_1) + 1$ as required.  Equality occurs when $d = N - n_1$ and we have equality in (\ref{binom}), which is equivalent to the conditions given in (\ref{equal1}).

In the case $m_1 = 2$, simplifying the definition of $\tau'_2(\cS)$ gives

\bes
\tau'_2(\cS) = \binom{N}{2} + \sum_{i \ge 2} n_i \binom{m_i +1}{2} + 3,
\ees
and after applying (\ref{binom}) we have

\bes
\tau'_2(\cS) \le \binom{N}{2} + \binom{N - 2 n_1 + 1}{2} + 3.
\ees
We must therefore show that

\bes
\binom{N}{2} + \binom{N - 2 n_1 + 1}{2} + 3 \le N(N - n_1) + 1.
\ees
Simplifying this gives

\begin{align*}
N(N-1) + (N - 2 n_1 + 1)(N - 2n_1) + 4 & \le 2N(N - n_1) \\
4 & \le 2Nn_1 - 4n_1^2 + 2n_1 \\
2 & \le n_1(N - 2n_1 + 1). \\
\end{align*}
The result now follows from $n_1 \ge 2$ and $N \ge n_1m_1 = 2n_1$, and equality occurs exactly in case (\ref{equal2}).

When $m_1 = 3$, after simplifying the definition of $\tau'_3(\cS)$ and dropping the $\epsilon$ term, we must show that

\bes
\binom{N}{2} + n_1 + 5 + \sum_{i \ge 2} n_i \binom{m_i +1}{2} \le \binom{N}{2} + n_1 + 5 + \binom{N - 3n_1 + 1}{2} \le N(N-n_1),
\ees
where the first inequality is (\ref{binom}).  Simplifying this gives

\begin{align*}
N(N-1) + 2n_1 + 10 + (N - 3n_1 + 1)(N - 3n_1) & \le 2N(N - n_1) \\
10 & \le n_1(4N - 9n_1 + 1).
\end{align*}
We have $n_1 \ge 2$ and $N \ge 3n_1$, so that $N \ge 6$ and $4N - 9n_1 + 1 \ge N+1 \ge 7$ as required.

\end{proof}

\section{Bounds for fixed vectors in representations of $GL_2$}\label{boundGL2}

Let $F$ be a $p$-adic field with residue field of order $q$, and residue characteristic different from 2 (i.e. $p\neq 2$).  Let $K_n$ be the standard principal congruence subgroups of $GL_2(F)$.  This section establishes the following bound that was used in the proof of Lemma \ref{sunramified}, both directly and as an ingredient in the proof of Theorem \ref{invdim}.

\begin{prop}
\label{GL2supercusp} Assume $p\neq 2$.
If $\pi$ is a supercuspidal representation of $GL_2(F)$, then $$\dim \pi^{K_n} \le q^n(1 + 1/q).$$
\end{prop}

\subsection{Review of supercuspidal representations}

We shall prove Proposition \ref{GL2supercusp} using the construction of supercuspidal representations of $GL_2(F)$ described in \cite[Section 7.A]{Ge}.  This produces a supercuspidal representation from a quadratic extension $L/F$ and a character $\chi$ of $L^\times$ that does not factor through the norm map to $F^\times$, and by Theorem 7.4 there, all supercuspidals are obtained in this way.

Fix such an extension $L$ and a character $\chi$.  Let $\cO$ and $\cO_L$ be the integer rings of $F$ and $L$, with maximal ideals $\p$ and $\p_L$, and let $\varpi$ and $\varpi_L$ be a choice of uniformizers.  We denote the conjugation of $L$ over $F$ by $x \mapsto \overline{x}$.  Let $N$ be the norm map from $L$ to $F$, and let $\omega$ be the non-trivial character of $F^\times / N(L^\times)$. Denote by $L^1$ the kernel of the norm map. Let $G_+$ be the subgroup of $GL_2(F)$ consisting of those $g$ with $\det(g) \in N(L^\times)$.  Let $\cS(L)$ be the space of Schwartz functions on $L$, and let $\cS(L)_\chi \subset \cS(L)$ be the subspace of functions satisfying $\Phi(xy) = \chi(y)^{-1}\Phi(x)$ for $y \in L^1$.

If $\tau$ is a non-trivial additive character of $F$, we may define a Fourier transform on $\cS(L)$ by setting $\langle x, y \rangle = \tau( \tr_{L/F}(xy))$ and defining
\[
\widehat{\Phi}(x) = \int_{L} \Phi(y) \langle x, y \rangle dy,
\]
where the Haar measure is normalised so that $\widehat{\widehat{\Phi}}(x) = \Phi(-x)$.  There is an irreducible representation $\pi_{\tau, \chi}$ of $G_+$ on $\cS(L)_\chi$ that satisfies

\begin{align}
\label{pitau1}
\pi_\tau\left( \left( \begin{array}{cc} 1 & u \\ 0 & 1 \end{array} \right) \right)\Phi(x) & = \tau( u N(x)) \Phi(x), \\
\label{pitau2}
\pi_\tau\left( \left( \begin{array}{cc} a & 0 \\ 0 & a^{-1} \end{array} \right) \right)\Phi(x) & = \omega(a) |a|^{1/2}_{L} \Phi(ax), \\
\label{pitau3}
\pi_\tau\left( \left( \begin{array}{cc} 0 & 1 \\ -1 & 0 \end{array} \right) \right)\Phi(x) & = \gamma \widehat{\Phi}(\overline{x}),
\end{align}
where $\gamma$ is a complex number of norm 1 depending only on $L$ and $\chi$, and $| \cdot |_L$ is the absolute value on $L$.  Note that the isomorphism class of $\pi_{\tau, \chi}$ depends only on $\chi$ and the orbit of $\tau$ under the action of $N(L^\times)$ on $\widehat{F} - \{ 0 \}$.  The supercuspidal $\pi$ associated to $L$ and $\chi$ is the induction of $\pi_{\tau,\chi}$ from $G_+$ to $GL_2(F)$.  Note that $\pi$ is independent of $\tau$, and if $\tau$ and $\tau'$ are representatives for the two orbits of $N(L^\times)$ on $\widehat{F} - \{ 0 \}$, the restriction of $\pi$ to $G_+$ is equal to $\pi_{\tau, \chi} \oplus \pi_{\tau', \chi}$.

We use this description of $\pi$ to bound $\dim \pi^{K_n}$.  We first observe that $K_n \subset G_+$ for all $n \ge 1$ by our assumption that $2 \nmid q$, and so $\dim \pi^{K_n} = \dim \pi_{\tau,\chi}^{K_n} + \dim \pi_{\tau',\chi}^{K_n}$.  We break the proof into the case where $L/F$ is ramified or unramified.

\subsection{The unramified case}

Let the conductor of $\tau$ be $\p^c$.

\begin{lemma}
\label{GL2unram1}

If $\Phi \in \pi_{\tau, \chi}^{K_n}$, then $\textup{supp}(\Phi) \subset \p_L^{\lceil (c-n)/2 \rceil}$ and $\Phi$ is constant on cosets of $\p_L^{c - \lceil (c-n)/2 \rceil}$.

\end{lemma}

\begin{proof}

If $x \in \text{supp}(\Phi)$ then (\ref{pitau1}) implies that $\tau( u N(x)) = 1$ for all $u \in \p^n$.  This implies that $N(x) \in \p^{c-n}$, or $x \in \p_L^{\lceil (c-n)/2 \rceil}$ as required.  For the second assertion, we note that $\left( \begin{array}{cc} 0 & 1 \\ -1 & 0 \end{array} \right)$ normalizes $K_n$, so by (\ref{pitau3}) we have $\pi_{\tau, \chi}\left( \left( \begin{array}{cc} 0 & 1 \\ -1 & 0 \end{array} \right) \right) \Phi = \gamma \widehat{\Phi}(\overline{\cdot}) \in \pi_{\tau, \chi}^{K_n}$.  Applying the first assertion, we find that $\text{supp}(\widehat{\Phi}) \subset \p_L^{\lceil (c-n)/2 \rceil}$.  The assertion now follows from the fact that if $y \in \p_L^a$, then the character $\langle x, y \rangle$ is constant on cosets of $\p_L^{c-a}$.

\end{proof}

\begin{lemma}
\label{GL2unram2}

We have $\dim \pi_{\tau,\chi}^{K_n} \le q^a$, where $a$ is equal to $n$ if $c-n$ is even and $n-1$ if $c-n$ is odd.

\end{lemma}

\begin{proof}

By combining Lemma \ref{GL2unram1} with the transformation property of any $\Phi \in \pi_{\tau, \chi}$ under $L^1$, we see that $\dim \pi_{\tau, \chi}^{K_n}$ is at most the number of orbits of $L^1$ on $\p_L^{\lceil (c-n)/2 \rceil} / \p_L^{c - \lceil (c-n)/2 \rceil}$, or equivalently the number of orbits on $\cO_L / \p_L^{c - 2 \lceil (c-n)/2 \rceil}$.  We note that $c - 2 \lceil (c-n)/2 \rceil = a$.  Using the decomposition
\[
\cO_L / \p_L^a = \{ \p_L^a \} \cup \bigcup_{t=1}^{a} \varpi^{a-t} [ (\cO_L / \p_L^{t} )^\times ],
\]
it may be seen that the number of orbits is
\[
1 + \sum_{t=1}^{a} \frac{ | (\cO_L / \p_L^{t} )^\times | }{ | L^1 / (L^1 \cap (1 + \p_L^{t} )) | } = 1 + \sum_{t=1}^{a} \frac{(q^2-1)q^{2(t-1)}}{(q+1)q^{t-1}} = 1+ \sum_{t=1}^{a} (q-1) q^{t-1} = q^a
\]
as required.

\end{proof}

Proposition \ref{GL2supercusp} now follows from Lemma \ref{GL2unram2}.  Indeed, if we let $\tau$ and $\tau'$ be representatives for the orbits of $N(L^\times)$ on $\widehat{F} - \{0 \}$, then the conductor exponents of $\tau$ and $\tau'$ must have opposite parities.  Combining $\dim \pi^{K_n} = \dim \pi_{\tau,\chi}^{K_n} + \dim \pi_{\tau',\chi}^{K_n}$ with Lemma \ref{GL2unram2} gives the proposition in this case.

\subsection{The ramified case} Again $c$ is the integer such that the conductor of $\tau$ is $\p^c$.

\begin{lemma}
If $\Phi \in \pi_{\tau, \chi}^{K_n}$, then $\textup{supp}(\Phi) \subset \p_L^{c-n}$ and $\Phi$ is constant on cosets of $\p_L^{c+n-1}$.
\end{lemma}

\begin{proof}

If $x \in \text{supp}(\Phi)$ then (\ref{pitau1}) implies that $\tau( u N(x)) = 1$ for all $u \in \p^n$.  This implies that $N(x) \in \p^{c-n}$, or $x \in \p_L^{c-n}$ as required.  For the second assertion, we have $\widehat{\Phi}(\overline{\cdot}) \in \pi_{\tau, \chi}^{K_n}$ as before, so that $\text{supp}(\widehat{\Phi}) \subset \p_L^{c-n}$.  The assertion now follows from the fact that if $y \in \p_L^a$, then the character $\langle x, y \rangle$ is constant on cosets of $\p_L^{2c-1-a}$.

\end{proof}

As in Lemma \ref{GL2unram2}, we must now count the orbits of $L^1$ on $\cO_L / \p_L^{2n-1}$.  We again write
\[
\cO_L / \p_L^{2n-1} = \{ \p_L^{2n-1} \} \cup \bigcup_{t=1}^{2n-1} \varpi_L^{2n-1-t} [ (\cO_L / \p_L^{t} )^\times ],
\]
so that the number of orbits is
\begin{align*}
1 + \sum_{t=1}^{2n-1} \frac{ | (\cO_L / \p_L^{t} )^\times | }{ | L^1 / (L^1 \cap (1 + \p_L^t )) | } & = 1 + \sum_{t=1}^{2n-1} \frac{(q-1)q^{t-1}}{2q^{\lfloor t/2 \rfloor}} \\
& = 1+ \frac{q-1}{2}(q^{n-1} + 2q^{n-2} + \ldots + 2 ) \\
& = (q^n + q^{n-1})/2.
\end{align*}
We therefore have $\dim \pi_{\tau, \chi}^{K_n} \le (q^n + q^{n-1})/2$, and the proposition now follows from $\dim \pi^{K_n} = \dim \pi_{\tau,\chi}^{K_n} + \dim \pi_{\tau',\chi}^{K_n}$.

\section{Bounds for fixed vectors in representations of $GL_3$}
\label{sec:Hoew}

Let $F$ be a $p$-adic field. Throughout this section we assume $p\neq 2,3$. Let $R$ be the ring of integers of $F$, $\varpi$ a uniformizer, $k$ the residue field, and $q$ its cardinality. Write $v:F^\times \to \Z$ for the additive valuation such that $v(\varpi)=1$. Let $G = GL_3(F)$, $K = GL_3(R)$, and $A = M_3(R)$.  Let $K_j$ be the subgroup of $K$ containing all elements congruent to 1 modulo $\varpi^j$. Put $U_j=1+\pi^j R$, a subgroup of $F^\times$. The aim of this section is to prove the following bound, which is used in the proof of Lemma \ref{sunramified} (thus in the proof of Proposition \ref{singlefinite}) when $m_1=3$.

\begin{theorem}
\label{invdim}
Assume $p\neq 2,3$.
If $\pi$ is an irreducible admissible representation of $G$, then
$$\dim \pi^{K_n} \le 9n^2 q^{4n} (1 + 1/q)^3.$$

\end{theorem}

Here is an outline of the proof. By Jacquet's subrepresentation theorem, $\pi$ is a submodule of a representation parabolically induced from something supercuspidal.  We therefore have three cases to consider.

\begin{itemize}

\item $\pi$ is a submodule of a representation $I$ induced from a character of the standard Borel $B$.  In this case, $\dim \pi^{K_n}  \le \dim I^{K_n} \le | K : K_n (B \cap K) | \le q^{3n} (1 + 1/q)^3$.

\item $\pi$ is a submodule of a representation $I$ induced from the standard parabolic of type (2,1) or (1,2).  Let $P$ be one of these parabolics, and let the representation of the Levi $GL_2 \times GL_1$ that we induce be $\pi' \otimes \chi$, where $\pi'$ is supercuspidal.  If $K'_n$ are the standard principal congruence subgroups of $GL_2$, we have $\dim \pi^{K_n} \le \dim I^{K_n} \le | K : K_n (P \cap K) | \dim \pi'^{K'_n}$.  Proposition \ref{GL2supercusp} gives $\dim \pi'^{K'_n} \le q^n (1 + 1/q)$, and we have $| K : K_n (P \cap K) | \le q^{2n} (1 + 1/q)^2$, so that $\dim \pi^{K_n} \le q^{3n} (1 + 1/q)^3$ as required.

\item $\pi$ is supercuspidal.  In this case, we apply the construction of supercuspidal representations of $GL_n(F)$ by Howe in \cite{Ho} when $n=3$.  It was shown in \cite{Moy} that these exhaust all supercuspidal representations of $G$ when $p$ is not 3.  This occupies the remainder of this section.
\end{itemize}

We may rewrite Theorem \ref{invdim} in a form which is less sharp, but better suited to the proof of our main theorem.

\begin{cor}\label{corinvdim}
Assume $p\neq 2,3$.
If $\pi$ is an irreducible admissible representation of $G$, then for every $\epsilon > 0$ there is a constant $C(q,\epsilon) > 0$ such that
\[
\dim \pi^{K_n} \le C(q,\epsilon) q^{(4+\epsilon)n}.
\]
Moreover, for any $\epsilon > 0$ there is $q(\epsilon) > 0$ such that we may take $C(q,\epsilon) = 1$ for all $q > q(\epsilon)$.

\end{cor}




\begin{remark}\label{qualitybound}
We may obtain a uniform bound on $\dim \pi^{K_n}$ of order $p^{8n}$ using the Plancherel theorem.  Indeed, if we let $Z_K$ be the center of $K$, and $\omega$ be the central character of $\pi$, we may define $f$ to be the function supported on $Z_K K_n$ and given by $f(z k) = \omega^{-1}(z)$.  Applying the Plancherel theorem to $f$ gives $\dim \pi^{K_n} \le d(\pi)^{-1} \mathrm{vol}(Z_K K_n)^{-1}$ for any Haar measure on $G$ and for any supercuspidal $\pi$, where $d(\pi)$ is the formal degree of $\pi$. By normalizing Haar measure, we can arrange that $d(\pi)$ is a positive integer, which gives $\dim \pi^{K_n} \ll p^{8n}$.  This is considered a trivial bound. On the other hand, for a fixed $\pi$ (either supercuspidal or any generic representation of $G$), the asymptotic growth of $\dim \pi^{K_n}$ is well known to be of order $q^{3n}$. (Such an asymptotic growth is known for general reductive $p$-adic groups either by character expansion or by a building argument \cite[Theorem 8.5]{MS12}.) So the bound $q^{4n}$ of Theorem \ref{invdim} is close to optimal (and more than enough for our global application). In fact, when applied to a fixed $\pi$, our method gives a bound of order $q^{3n}$. It is an interesting question whether a uniform bound of order $q^{3n}$ holds (or for a general reductive group, whether a uniform bound can be established to the same order as the bound for an individual representation).

\end{remark}

\subsection{An overview of Howe's construction}
\label{overview}

We now describe the construction of Howe in more detail, including the features we shall use to prove Theorem \ref{invdim}.  Howe's representations $V(\psi')$ are parametrized by a degree 3 extension $F'/F$ and a character $\psi'$ of $F'^\times$, satisfying a condition called admissibility.  Fix such an $F'$ and $\psi'$, and let $R'$, $\varpi'$, and $k'$ be the ring of integers, uniformizer, and residue field of $F'$. Write $N=N(F'/F)$ for the norm map from $F'$ to $F$. Choose a basis for $R'$ as a free $R$-module, which defines an embedding $F' \subset M_3(F)$.  We shall identify $F'$ with a subalgebra of $M_3(F)$ from now on.  We define the order $A' = \cap_{x \in F'^\times} x A x^{-1}$, which is characterized as the set of matrices $M$ such that $M \varpi'^i R^3 \subset \varpi'^i R^3$ for all $i$.  We define $K' = \cap_{x \in F'^\times} x K x^{-1} = A' \cap GL_3(R)$, which is the subgroup of matrices preserving the lattices $\varpi'^i R^3$.  For $i \ge 1$ we define $K'_i = 1 + \varpi'^i A'$ and $U'_i = 1 + \varpi'^i R'$.  Let $j$ be the conductor of $\psi'$, that is the minimal $j$ such that $\psi'$ is trivial on $U'_j$.  The admissibility condition placed on $\psi'$ implies that $j \ge 1$.

In \cite[Lemma 12]{Ho}, Howe constructs a representation $W(\psi')$ of $K' F^{\prime\times}$\footnote{Lemma 12 only defines $W(\psi')$ as a representation of $K'$, but it can be extended to $K' F'^\times$ by the remarks at the start of \cite[Thm 2]{Ho}.}, and defines the supercuspidal representation $V(\psi')$ associated to $F'$ and $\psi'$ to be the compact induction of $W(\psi')$ to $G$.  We know that $\dim \pi^{K_n}$ is at most $\dim W(\psi')$ times the number of double cosets of the form $K' F'^\times g K_n$ that support $K_n$-invariant vectors, that is such that $W(\psi')^{g K_n g^{-1} \cap K' F'^\times} \neq 0$.  Bounding $\dim W(\psi')$ is easy, while bounding the number of these double cosets requires a feature of $W(\psi')$ from Howe's paper that we now describe.

We first assume that $j \ge 2$.  The representation $W(\psi')$ is trivial on $K'_j$, and because $K'_{j-1} / K'_j$ is abelian, $W(\psi')|_{K'_{j-1}}$ decomposes into characters.  Howe defines a character $\psi$ of $K'_{j-1} / K'_j$ by taking the natural extension of $\psi'$ from $U'_{j-1}$, and shows that $W(\psi')|_{K'_{j-1}}$ contains exactly the characters lying in the $K'$-orbit of $\psi$ for the natural action of $K'$ on $\widehat{K'_{j-1} / K'_j}$.

We use this fact to control those $g$ supporting invariant vectors by observing that if $W(\psi')^{g K_n g^{-1} \cap K' F'^\times} \neq 0$, then $W(\psi')^{g K_n g^{-1} \cap K'_{j-1} } \neq 0$.  However, if $g \in K \lambda(\varpi) K$ with $\lambda \in X_*(T)$ too large, then $g K_n g^{-1} \cap K'_{j-1}$ will contain the intersection of $K'_{j-1}$ with a unipotent subgroup of $G$, and this will turn out to be incompatible with the description of $W(\psi')|_{K'_{j-1}}$.  In the case $j = 1$, $F'$ is unramified over $F$ and $W(\psi')$ is inflated from a cuspidal representation of $K' / K'_1 \simeq GL_3(k)$, and we may use this to argue in a similar way.

In the case of $GL_3$, Howe's construction may be naturally divided into the cases where $F'/F$ is ramified or unramified.  We shall therefore divide our proof into these two cases, after introducing some more notation and defining the character $\psi$.

\subsection{The character $\psi$}
\label{psidef}

We now assume that $j \ge 2$, and define the character $\psi$ of $K'_{j-1} / K'_j$ that appears in the description of $W(\psi')|_{K'_{j-1}}$.

Let $B'$ be the group of prime to $p$ roots of unity in $F'^\times$, which is naturally identified with $k'^\times$.  Let $C'$ be the group generated by $B'$ and $\varpi'$.  Let $\langle \; , \; \rangle$ be the pairing $\langle S, T \rangle = \tr(ST)$ on $M_3(F)$.  Let $\chi$ be a character of $F$ of conductor $R$, which defines an isomorphism $\theta : M_3(F) \to \widehat{M_3(F)}$ by $\theta(S)(T) = \chi( \langle S, T \rangle )$.  Let $e$ denote the degree of ramification of $F' / F$.  Because the dual lattice to $A'$ under $\langle , \rangle$ is $\varpi'^{1 - e} A'$ by \cite[Lemma 2]{Ho}, the map $\theta$ gives an isomorphism between the character group of $\varpi'^{i-1} A' / \varpi'^i A'$ and $\varpi'^{-i -e +1} A' / \varpi'^{-i - e +2} A'$.  We may combine $\theta$ with the isomorphism $K'_{j-1} / K'_j \simeq \varpi'^{j-1} A' / \varpi'^j A'$ to obtain $\mu : \widehat{K'_{j-1} / K'_j} \to \varpi'^{-j-e+1} A' / \varpi'^{-j-e+2} A'$.  If $\mu(\varphi) = y + \varpi'^{-j-e+2} A'$, we say that $y$ represents $\varphi$.  If $\varphi$ has a representative $y \in F'^\times$, we see that $\varphi$ also has a unique representative $c \in C'$, which is called the standard representative of $\varphi$.

The map $\theta$ restricts to a map $F' \to \widehat{F'}$, which is also given by $\theta(x)(y) = \chi( \tr_{F'/F} xy )$.  We may combine $\theta$ with the isomorphism $U'_{j-1} / U'_j \simeq \varpi'^{j-1} R' / \varpi'^j R'$ to obtain $\mu' : \widehat{U'_{j-1} / U'_j} \to \varpi^{-j-e+1} R' / \varpi^{-j-e+2} R'$.  If $\mu'(\varphi) = y + \varpi^{-j-e+2} R'$, we say that $y$ represents $\varphi$.  We see that a nontrivial $\varphi$ has a unique representative $c \in C'$, which is called the standard representative of $\varphi$.

We define $\psi$ by taking the standard representative $c$ for $\psi'$ on $U'_{j-1}$, and letting $\psi$ be the character represented by $c$.  If we let $\Ad^*$ denote the natural action of $K'$ on $\widehat{K'_{j-1} / K'_j}$, given explicitly by $[ \Ad^*(k)\psi](g) = \psi(k^{-1} g k)$, then \cite[Lemma 12]{Ho} states that $W(\psi')|_{K'_{j-1}}$ contains exactly the characters in $\Ad^*(K') \psi$.

\subsection{Reduction to the case $c \notin F$}
\label{reduction}

We may carry out the argument sketched in Section \ref{overview} once we have reduced to the case where either $j = 1$ or $c \notin F$.  We perform this reduction by observing that if $j \ge 2$ and $c \in F$, then $V(\psi')$ is a twist of $V(\psi_1)$ for some $\psi_1$ of smaller conductor.  Indeed, by \cite[Lemma 11]{Ho}, if $c \in F$ then we may write $\psi' = \psi_1 \psi_2$, where $\psi_1$ is trivial on $U'_{j-1}$ and $\psi_2 = \psi'' \circ N(F'/F)$ for some character $\psi''$ of $F^\times$.

\begin{lemma}

We have $V(\psi') = V(\psi_1) \otimes \psi'' \circ \det$.

\end{lemma}

\begin{proof}

This follows by examining the construction of $W(\psi')$ in \cite[Lemma 12]{Ho}.  In the case $c \in F$, the groups $H_i$ defined by Howe are equal to $K'_i$, and the group $GL_l(F'')$ appearing in the proof of \cite[Lemma 12]{Ho} is equal to $GL_3(F)$.  Howe constructs $W(\psi')$ by taking the representation $W(\psi_1)$ of $K'$ associated to $\psi_1$ (which he denotes $W''(\psi_1')$, and whose construction can be assumed as $\psi_1$ has smaller conductor) and forming the twist $W(\psi_1) \otimes \psi'' \circ \det$.  He then obtains $W(\psi')$ by applying the correspondence of \cite[Thm 1]{Ho} to $W(\psi_1) \otimes \psi'' \circ \det$, which in this case is trivial so that $W(\psi') = W(\psi_1) \otimes \psi'' \circ \det$.  As $V(\psi')$ and $V(\psi_1)$ are the inductions of $W(\psi')$ and $W(\psi_1)$, the lemma follows.

\end{proof}

The next lemma shows that it suffices to consider $V(\psi_1)$.

\begin{lemma}

We have $\dim V(\psi')^{K_n} \le \dim V(\psi_1)^{K_n}$.

\end{lemma}

\begin{proof}

Because $N(U'_i) = U_{\lceil i/e \rceil}$, if a character $\varphi$ of $F^\times$ has conductor $i+1$, then $\varphi \circ N(F'/F)$ has conductor $ei+1$.  Because $\psi'' \circ N$ has conductor $j \ge 2$, this implies that there is some $i \ge 1$ such that $j = ei + 1$ and $\psi''$ has conductor $i+1$.

If $n \ge i+1$ then $\det K_n \subset U_{i+1}$.  This implies that $\psi'' \circ \det$ is trivial on $K_n$, which gives the lemma.  Suppose that $n \le i$ and $V(\psi')^{K_n} \neq 0$.  As the central character of $V(\psi')$ is $\psi'|_{F^\times}$, this implies that $\psi'$ is trivial on $U_n$, and hence on $U_i$.  As $\psi_1$ is trivial on $U'_{j-1} \cap F = U_i$, this implies that $\psi_2 = \psi'' \circ N(F'/F)$ is trivial on $U_i$.  This implies that $\psi''$ is trivial on $U_i$, which contradicts it having conductor $i+1$.

\end{proof}

By replacing $\psi'$ with $\psi_1$, multiple times if necessary, we may assume for the rest of the proof that $j = 1$ or $c \notin F$.

\subsection{The unramified case}

Here, the groups $K'$ and $K_j'$ are equal to $K$ and $K_j$ respectively, and so we omit the $'$ in this section.  We also take $\varpi' = \varpi$.  The embedding $F' \subset M_3(F)$ has the property that $R' = F' \cap M_3(R)$, so that it induces an embedding $k' \subset M_3(k)$.  It also satisfies $R'^\times = F'^\times \cap K$ and $U_i' = F'^\times \cap K_i$.

We need to bound $\dim W(\psi')$, and the number of double cosets $K F'^\times g K_n$ such that $W(\psi')^{g K_n g^{-1} \cap K F'^\times} \neq 0$, and we begin with the second problem.  We note that $F'^\times \subset K Z$ in the unramified case, where $Z$ is the center of $G$, so that $K F'^\times = K Z$.  The dimension of $W(\psi')^{g K_n g^{-1} \cap K Z}$ depends only on the double coset $K Z g K$.  By the Cartan decomposition, we may therefore break the problem into finding those $\lambda \in X_*(T)^+$ such that $ZK \lambda(\varpi) K$ supports invariant vectors, where $T$ is the diagonal torus in $G$, and then count the $(ZK, K_n)$-double cosets in a given $ZK \lambda(\varpi) K$.  These steps are carried out by Lemmas \ref{lambdabound} and \ref{doublecoset}.  We write $\lambda \in X_*(T)$ as $(\lambda_1, \lambda_2, \lambda_3)$.

\begin{lemma}
\label{lambdabound}

If $\lambda \in X_*(T)^+$ is such that $W(\psi')^{\lambda(\varpi) K_n \lambda(\varpi)^{-1} \cap K Z} \neq 0$, then $\max \{ \lambda_1 - \lambda_2, \lambda_2 - \lambda_3 \} \le n-j$.
(It follows that $j\le n$.)

\end{lemma}

\begin{proof}

We will prove $\lambda_1-\lambda_2\le n-j$ by contradiction; the argument for $\lambda_2 - \lambda_3$ is exactly analogous.

First we treat the case $j>1$.
The hypothesis implies that $W(\psi')^{\lambda(\varpi) K_n \lambda(\varpi)^{-1} \cap K_{j-1} } \neq 0$, and we use the description of $W(\psi')$ restricted to $K_{j-1}$.  We identify $K_{j-1} / K_j \simeq M_3(k)$.  If $c = b \varpi^{-j}$ with $b \in B'$, and we identify $b$ with an element of $k'^\times \subset M_3(k)$, then $\psi|_{K_{j-1}}$ under this identification corresponds to the character of $M_3(k)$ given by  $x \to \chi( \tr(bx) / \varpi )$.

If $\lambda_1 - \lambda_2 \ge n-j+1$, a simple calculation shows that the image of $\lambda(\varpi) K_n \lambda(\varpi)^{-1} \cap K_{j-1}$ in $K_{j-1} / K_j \simeq M_3(k)$ contains the subgroup

\bes
Y = \left( \begin{array}{ccc}
&& \\
* && \\
* &&
\end{array} \right) \subset M_3(k).
\ees
There must be a character in the orbit $\Ad^* K(\psi)$ which is trivial on $Y$, which means that there is $k \in K$ such that $\tr( y \Ad(k)b ) = 0$ for all $y \in Y$.  The annihilator of $Y$ under the trace pairing is

\bes
Y^\perp = \left( \begin{array}{ccc}
* && \\
* & * & * \\
* & * & *
\end{array} \right) \subset M_3(k),
\ees
so that $\Ad(k) b \in Y^\perp$.  Any $y \in Y^\perp$ has eigenvalues that lie in the quadratic extension of $k$, while the eigenvalues of $b$ lie in $k' - k$ because $c \notin F$, which is a contradiction.

Next we consider the case $j=1$.  In this case, the proof of \cite[Lemma 12]{Ho} states that $W(\psi')$ is inflated from a cuspidal representation of $GL_3(k)$.  If $\lambda_1 - \lambda_2 \ge n$, then the image of $\lambda(\varpi) K_n \lambda(\varpi)^{-1} \cap K$ in $K / K_1 \simeq GL_3(k)$ contains the subgroup

\bes
Y = \left( \begin{array}{ccc}
1 && \\
* & 1 & \\
* && 1
\end{array} \right) \subset GL_3(k).
\ees
However, a cuspidal representation cannot have any vectors invariant under $Y$, because then it would be a subrepresentation of a representation induced from a parabolic of type (2,1).

\end{proof}

\begin{lemma}
\label{doublecoset}

Let $\lambda \in X_*(T)^+$ satisfy $\max \{ \lambda_1 - \lambda_2, \lambda_2 - \lambda_3 \} \le n-j$ as in Lemma \ref{lambdabound}.  The number of $(ZK, K_n)$-double cosets in $ZK \lambda(\varpi) K$ is at most $q^{4n-4j}(1 + 1/q)^3$.

\end{lemma}

\begin{proof}

Any double coset $ZK g K_n$ contained in $ZK \lambda(\varpi) K$ has a representative with $g \in \lambda(\varpi) K$.  It may be seen that $\lambda(\varpi) k_1$ and $\lambda(\varpi) k_2$ represent the same double coset if and only if $k_1 \in \lambda(\varpi)^{-1} K \lambda(\varpi) k_2 K_n$, and so if we define $K_\lambda = \lambda(\varpi)^{-1} K \lambda(\varpi) \cap K$ then the number of double cosets is equal to $| K_\lambda \backslash K / K_n |$.

If $j=n$ then $\lambda_1=\lambda_2=\lambda_3$ so $K_\lambda=K$, hence the lemma is trivial. So we may assume $j\le n-1$.
Then $K_\lambda$ contains any matrix $g=(g_{a,b}) \in M_3(R)$ such that $v(g_{2,1})\ge n-j$, $v(g_{3,2})\ge n-j$, $v(g_{3,1})\ge 2n-2j$, and $v(g_{i,i})=0$ for $1\le i\le 3$ (the last condition ensures that $g\in K$ as $g$ mod $\varpi$ is upper triangular).
This implies that the image of $K_\lambda$ in $K/K_n\simeq GL_3(R/\varpi^n)$ has cardinality at least $q^{5n+4j} (1 - 1/q)^3$. Therefore

\begin{align*}
| K_\lambda \backslash K / K_n | & \le | K / K_n | / | \mathrm{image}(K_\lambda) | \\
& \le q^{9n}(1 - q^{-3})(1 - q^{-2})(1 - q^{-1}) / q^{5n+4j}(1 - 1/q)^3 \\
& = q^{4n-4j}(1 + q^{-1} + q^{-2})(1 + q^{-1}) \le q^{4n-4j}(1 + 1/q)^3.
\end{align*}

\end{proof}

Lemma \ref{lambdabound} implies that there are at most $n^2$ choices of $\lambda \in X^+_*(T) / X_*(Z)$ such that $KZ \lambda(\varpi) K$ supports invariant vectors, and combining this with Lemma \ref{doublecoset} shows that there are at most $n^2 q^{4n-4j} (1+1/q)^3$ double cosets $KZ g K_n$ that support invariant vectors.  This gives
$$\dim V(\psi')^{K_n} \le n^2 q^{4n-4j} (1+1/q)^3 \dim W(\psi').$$

We now bound $\dim W(\psi')$.  We first assume that $j \ge 2$.  Our assumption that $c \notin F$ implies that the field $F''$ in \cite[Lemma 12]{Ho} is the same as $F'$, and the groups $H_i$ are given by $H_0 = R'^\times$ and $H_i = U'_i$ for $i \ge 1$.  Following the proof of that lemma, we see that $W(\psi')$ is the representation associated to the character $\psi'$ on $R'^\times$ by \cite[Thm 1]{Ho}.  When $j$ is even, that theorem implies that $W(\psi')$ is the induction of a character of $R'^\times K_{j/2}$ to $K$, so that $\dim W(\psi') = | K : R'^\times K_{j/2} |$.  We have $| K : R'^\times K_{j/2} | = q^{3j}(1 - 1/q)(1 - 1/q^2) \le q^{3j}$.

When $j$ is odd, we let $j = 2i+1$.  The construction of $W(\psi')$ in this case is described on \cite[p. 448]{Ho}, and is given by inducing a representation $J$ from $R'^\times K_i$ to $K$.  The discussion on p. 448 implies that $J$ has the same dimension as the two representations denoted $V(\widetilde{\varphi}'')$ and $V(\psi)$ there, and Howe states that $\dim V(\psi) = ( \# \widetilde{\mathcal{H}} / \widetilde{\mathcal{Z}} )^{1/2}$ for two groups $\widetilde{\mathcal{H}}$ and $\widetilde{\mathcal{Z}}$.  Moreover, on p. 447 he states that $\widetilde{\mathcal{H}} / \widetilde{\mathcal{Z}} \simeq \mathcal{H} / \mathcal{Z} \simeq K_i / U_i' K_{i+1}$.  As $i \ge 1$, we have $\dim J = | K_i / U_i' K_{i+1} |^{1/2} = q^3$.  We then have $\dim W(\psi') = q^3 | K : R'^\times K_i | = q^3 q^{6i} (1 - 1/q)(1 - 1/q^2) \le q^{6i + 3} = q^{3j}$.

In the remaining case $j = 1$, $W(\psi')$ is inflated from a cuspidal representation of $GL_3(k)$. Such a cuspidal representation has dimension $(q^2-1)(q-1)$. Therefore $\dim W(\psi')=(q^2-1)(q-1)\le q^3 = q^{3j}$.

In all cases we have verified $\dim W(\psi') \le q^{3j}$. Hence
$$\dim V(\psi')^{K_n} \le n^2 q^{4n-4j} (1+1/q)^3 \cdot q^{3j} \le n^2 q^{4n} (1+1/q)^3 .$$

\subsection{The ramified case}

In this case we must have $j \ge 2$.  Moreover, $F'$ is tamely ramified over $F$ (since $p\neq 3$) and generated by a cube root of a uniformizer $\varpi$ of $F$. Thus we may assume $\varpi^{\prime 3}=\varpi$. Choose our basis for $R'$ as a free $R$-module to be $\{1, \varpi', \varpi'^2 \}$.  With this choice, the image of $\varpi'$ in $G$ is

\bes
\varpi' = \left( \begin{array}{ccc} && \varpi \\ 1 && \\ & 1 & \end{array} \right).
\ees
We see that $\varpi'^i A'$ is given by
\begin{align}
\label{Aprime1}
\varpi'^{3i}A' & = \varpi^i \left( \begin{array}{ccc} * & \varpi * & \varpi * \\ * & * & \varpi * \\ * & * & * \end{array} \right), \\
\label{Aprime2}
\varpi'^{3i+1} A' & = \varpi^i \left( \begin{array}{ccc} \varpi * & \varpi * & \varpi * \\ * & \varpi * & \varpi * \\ * & * & \varpi * \end{array} \right), \\
\label{Aprime3}
\varpi'^{3i+2} A' & = \varpi^i \left( \begin{array}{ccc} \varpi * & \varpi * & \varpi^2 * \\ \varpi * & \varpi * & \varpi * \\ * & \varpi * & \varpi * \end{array} \right),
\end{align}
where the *'s lie in $R$.  As $K' = A'^\times$, $K'$ is the lower triangular Iwahori subgroup.

The proof may be naturally broken into cases depending on the residue class of $j$ modulo 3.  We may assume that $j \not\equiv 1 \; (3)$, as in this case we have $c \in F$.  Note that we are using our assumption that $\varpi'^3 = \varpi$ here.

As in the unramified case, we begin by observing that it suffices to bound $\dim W(\psi')$ and the number of double cosets $K' F'^\times g K_n$ such that $W(\psi')^{g K_n g^{-1} \cap K' F^\times} \neq 0$.  This condition depends only on $K' F'^\times g K$, and the following lemma gives a convenient set of representatives for these double cosets.

\begin{lemma}

If $\Sigma = \{ \lambda \in X_*(T) : \lambda_1 + \lambda_2 + \lambda_3 = 0 \}$, we have $G = \bigcup_{\lambda \in \Sigma} K' F'^\times \lambda(\varpi) K$.

\end{lemma}

\begin{proof}

We use the Bruhat decomposition.  Let $T^1$ and $N(T)$ be the maximal compact subgroup and normalizer of $T$.  We define the Weyl group $W = N(T) / T$ and affine Weyl group $\widetilde{W} = N(T) / T^1$.  We identify $W$ with the group of permutation matrices in $K$, and hence with a subgroup of $\widetilde{W}$.  We then have $\widetilde{W} \simeq X_*(T) \rtimes W$, and $\widetilde{W}$ may be identified with matrices of the form $\lambda(\varpi) w$ with $\lambda \in X_*(T)$ and $w \in W$.

We have $\varpi' \in N(T)$, and it may be seen that $\widetilde{W} = \langle \varpi' \rangle \Sigma W$.  Indeed, the action of $\varpi'$ on $\widetilde{W} / W \simeq X_*(T)$ by left multiplication is given by

\[
\varpi'(\lambda_1, \lambda_2, \lambda_3) = (\lambda_3 +1, \lambda_1, \lambda_2),
\]
so every orbit contains a unique element of $\Sigma$.  The Bruhat decomposition then gives

\[
G = \bigcup_{w \in \widetilde{W} } K' w K' = \bigcup_{\lambda \in \Sigma} K' \langle \varpi' \rangle \lambda(\varpi) W K' = \bigcup_{\lambda \in \Sigma} K' F'^\times \lambda(\varpi) K
\]
as required.

\end{proof}

The next lemma bounds those $\lambda \in \Sigma$ such that $K' F'^\times \lambda(\varpi) K$ supports invariant vectors.

\begin{lemma}
\label{ramlambdabound}

If $\lambda \in \Sigma$ satisfies $W(\psi')^{\lambda(\varpi) K_n \lambda(\varpi)^{-1} \cap K' F'^\times} \neq 0$, then

\begin{align}
\label{ramlambda1}
\max\{ \lambda_1 - \lambda_2, \lambda_2 - \lambda_3, \lambda_3 - \lambda_1 +1 \} & \le n-i-1 \quad \text{if} \quad  j = 3i + 2, \\
\label{ramlambda2}
\max\{ \lambda_2 - \lambda_1, \lambda_3 - \lambda_2, \lambda_1 - \lambda_3 - 1 \} & \le n-i-1 \quad \text{if} \quad  j = 3i.
\end{align}
In either case, summing the three bounds gives $j \le 3n-2$.

\end{lemma}

\begin{proof}

We may naturally identify $K' / K'_1$ and $K'_{j-1} / K'_j$ with $(k^\times)^3$ and $k^3$ using the coordinate entries in such a way that the adjoint action of $K' / K'_1$ on $K'_{j-1} / K'_j$ is given by

\[
\Ad(x_1, x_2, x_3)(y_1, y_2, y_3) = \bigg\{ \begin{array}{ll} (x_1 x_2^{-1} y_1, x_2 x_3^{-1} y_2, x_3 x_1^{-1} y_3), & j \equiv 0 \; (3), \\
(x_1^{-1} x_2 y_1, x_2^{-1} x_3 y_2, x_3^{-1} x_1 y_3), & j \equiv 2 \; (3). \end{array}
\]
Moreover, if $c$ is equal to $\varpi'^{-j-2} b$ with $b \in B' \simeq k^\times$, then the character $\psi$ of $K'_{j-1}/K'_j$ is given by $\psi(y_1, y_2, y_3) = \chi( b(y_1 + y_2 + y_3) / \varpi)$.  This implies that $\Ad^*(h) \psi$ is nontrivial on every coordinate subgroup in $K'_{j-1} / K'_j \simeq k^3$ for any $h \in K'$.

If $W(\psi')^{\lambda(\varpi) K_n \lambda(\varpi)^{-1} \cap K' F'^\times} \neq 0$, then $W(\psi')^{\lambda(\varpi) K_n \lambda(\varpi)^{-1} \cap K'_{j-1}} \neq 0$.  Because $W(\psi')|_{K'_{j-1}}$ is a sum of characters of the form $\Ad^*(h) \psi$ with $h \in K'$, one such character must be trivial on $\lambda(\varpi) K_n \lambda(\varpi)^{-1} \cap K'_{j-1}$, which implies that the image of $\lambda(\varpi) K_n \lambda(\varpi)^{-1} \cap K'_{j-1}$ in $K'_{j-1} / K'_j$ does not contain a coordinate subgroup.  By combining the definition $K'_{j-1} = 1 + \varpi'^{j-1}A'$ with (\ref{Aprime1})--(\ref{Aprime3}), we see that this implies the inequalities (\ref{ramlambda1}) and (\ref{ramlambda2}).

\end{proof}

\begin{lemma}
\label{ramdoublecoset}

Let $\lambda \in \Sigma$ satisfy (\ref{ramlambda1}) or (\ref{ramlambda2}).  The number of $(K' F'^\times, K_n)$-double cosets in $K' F'^\times \lambda(\varpi) K$ is at most $q^{3n-2i} (1 + 1/q)^3$ when $j = 3i$, and $q^{3n-2i-2} (1 + 1/q)^3$ when $j = 3i+2$.

\end{lemma}

\begin{proof}

Any double coset $K' F'^\times g K_n$ contained in $K' F'^\times \lambda(\varpi) K$ has a representative of the form $\lambda(\varpi) k$, and two elements $\lambda(\varpi) k_1$ and $\lambda(\varpi) k_2$ represent the same double coset if and only if $k_2 \in \lambda(\varpi)^{-1} K' F'^\times \lambda(\varpi) k_1 K_n$.  Therefore if we define $K'_\lambda = \lambda(\varpi)^{-1} K' \lambda(\varpi) \cap K$, the number of double cosets is bounded by $| K'_\lambda\backslash K / K_n |$.  It may be seen that $K_\lambda'$ contains any matrix $g=(g_{a,b}) \in M_3(R)$ satisfying the conditions

\begin{align*}
v(g_{a,b}) & \ge \max\{ \lambda_b - \lambda_a + 1, 0 \}, \quad a < b \\
v(g_{i,i}) & =0, \quad 1 \le i \le 3 \\
v(g_{a,b}) & \ge \max\{ \lambda_b - \lambda_a, 0 \}, \quad a > b
\end{align*}
on the upper triangular, diagonal, and lower triangular entries respectively.  The reader should note that for each pair $a \neq b$, one may order $(a,b)$ so that the inequalities on $v(g_{a,b})$ and $v(g_{b,a})$ have the form $v(g_{a,b}) \ge c$, $v(g_{b,a}) \ge 0$ for some $c > 0$.  Moreover, the set of entries for which the inequality above reads $v(g_{a,b}) \ge 0$ form the unipotent radical of a Borel subgroup $B_\lambda$ containing the diagonal matrices.  It follows that $K'_\lambda$ lies between $B_\lambda \cap GL_3(R)$ and $(B_\lambda \cap GL_3(R)) K_1$.

We will divide the proof into six cases depending on the possibilities for $B_\lambda$.  We treat one case in detail, and describe the modifications to be made in the others.
\medskip

\noindent
{\bf Case 1:} $\displaystyle B_\lambda = \left( \begin{array}{ccc} * & * & * \\ & * & * \\ && * \end{array} \right)$
\medskip

\noindent
In this case, the significant congruence conditions imposed on $g = (g_{a,b}) \in K_\lambda'$ are

\be
\label{entryconds1}
v(g_{2,1}) \ge \lambda_1 - \lambda_2, \quad v(g_{3,1}) \ge \lambda_1 - \lambda_3, \quad \text{and} \quad v(g_{3,2}) \ge \lambda_2 - \lambda_3.
\ee
The image of $K_\lambda'$ in $K / K_n$ therefore has cardinality at least $q^{9n + 2\lambda_3 - 2\lambda_1}(1 - 1/q)^3$, and so as in Lemma \ref{doublecoset} we have $| K_\lambda' \backslash K / K_n | \le | K / K_n | / | \mathrm{image}(K_\lambda') | \le (1+1/q)^3 q^{2\lambda_1 - 2\lambda_3}$.  If $j = 3i$ then Lemma \ref{ramlambdabound} gives $2\lambda_1 - 2\lambda_3 \le 2n - 2i$ as required.  If $j = 3i+2$, Lemma \ref{ramlambdabound} does not provide a strong enough bound for $\lambda_1 - \lambda_3$, and so we instead observe that the image of $K_\lambda'$ in $K/K_n$ contains those matrices satisfying

\be
\label{entryconds2}
v(g_{2,1}) \ge \lambda_1 - \lambda_2, \quad v(g_{3,1}) \ge n, \quad \text{and} \quad v(g_{3,2}) \ge \lambda_2 - \lambda_3,
\ee
with the other conditions unchanged.  This group has cardinality at least $q^{8n + \lambda_3 - \lambda_1} (1-1/q)^3$, and so $| K_\lambda' \backslash K / K_n | \le (1+1/q)^3 q^{n + \lambda_1 - \lambda_3}$.  Lemma \ref{ramlambdabound} gives

\[
\lambda_1 - \lambda_3 = (\lambda_1 - \lambda_2) + (\lambda_2 - \lambda_3) \le 2n - 2i -2,
\]
which gives the Lemma in this case.

In the other five cases, we may apply the same method to produce a bound of the form $| K_\lambda' \backslash K / K_n | \le (1+1/q)^3 q^\tau$, where $\tau$ depends on the residue class of $j$ modulo 3.  We describe the underlying recipe for finding $\tau$ in the case above, and then show what it gives in each remaining case.  When $j \equiv 0 \; (3)$, we added the right hand sides of (\ref{entryconds1}), and the resulting expression $2(\lambda_1 - \lambda_3)$ could be bounded using one application of Lemma \ref{ramlambdabound}, which gave $\tau$.  When $j \equiv 2 \; (3)$, we modified the bound in (\ref{entryconds1}) corresponding to the non-simple positive root for $B_\lambda$ to obtain (\ref{entryconds2}), added the right hand sides, and bounded the result using two applications of Lemma \ref{ramlambdabound} to give $\tau$.

We now find $\tau$ in the remaining 5 cases, and check that
\[
\tau \le \bigg\{ \begin{array}{ll} 3n-2i-2, & j = 3i+2, \\ 3n-2i, & j = 3i. \end{array}
\]
Note that in some cases we may need to use the assumption that $n \ge 1$, which we are free to make.
\medskip

\noindent
{\bf Case 2:} $\displaystyle B_\lambda = \left( \begin{array}{ccc} * & * & * \\ & * & \\ &*& * \end{array} \right)$
\medskip

\noindent
The analog of (\ref{entryconds1}) is

\bes
v(g_{2,1}) \ge \lambda_1 - \lambda_2, \quad v(g_{3,1}) \ge \lambda_1 - \lambda_3, \quad v(g_{2,3}) \ge \lambda_3 - \lambda_2 + 1,
\ees
which is modified to $v(g_{2,1}) \ge n$.  We have

\begin{align*}
2\lambda_1 - 2\lambda_2 +1 \le 2n-2i-1 & = \tau \quad \text{when} \quad j = 3i+2, \\
n+\lambda_1-\lambda_2+1 \le 3n-2i & = \tau \quad \text{when} \quad j = 3i.
\end{align*}

\noindent
{\bf Case 3:} $\displaystyle B_\lambda = \left( \begin{array}{ccc} * & & * \\ *& * &* \\ && * \end{array} \right)$
\medskip

\noindent
The analog of (\ref{entryconds1}) is

\bes
v(g_{1,2}) \ge \lambda_2 - \lambda_1+1, \quad v(g_{3,1}) \ge \lambda_1 - \lambda_3, \quad v(g_{3,2}) \ge \lambda_2 - \lambda_3,
\ees
which is modified to $v(g_{3,2}) \ge n$.  We have

\begin{align*}
2\lambda_2 - 2\lambda_3+1 \le 2n-2i-1 & = \tau \quad \text{when} \quad j = 3i+2, \\
n+\lambda_2 - \lambda_3+1 \le 3n-2i & = \tau \quad \text{when} \quad j = 3i.
\end{align*}

\noindent
{\bf Case 4:} $\noindent B_\lambda = \left( \begin{array}{ccc} * & * & \\ & * & \\ *&*& * \end{array} \right)$
\medskip

\noindent
The analog of (\ref{entryconds1}) is

\bes
v(g_{2,1}) \ge \lambda_1 - \lambda_2, \quad v(g_{1,3}) \ge \lambda_3 - \lambda_1+1, \quad v(g_{2,3}) \ge \lambda_3 - \lambda_2+1,
\ees
which is modified to $v(g_{2,3}) \ge n$. We have

\begin{align*}
2\lambda_3 - 2\lambda_2+2 \le 2n-2i & = \tau \quad \text{when} \quad j = 3i, \\
n+\lambda_3 - \lambda_2 +1 \le 3n-2i-2 & = \tau \quad \text{when} \quad j = 3i+2.
\end{align*}

\noindent
{\bf Case 5:} $\displaystyle B_\lambda = \left( \begin{array}{ccc} * & & \\ *& * & * \\ *&& * \end{array} \right)$
\medskip

\noindent
The analog of (\ref{entryconds1}) is

\bes
v(g_{1,2}) \ge \lambda_2 - \lambda_1+1, \quad v(g_{1,3}) \ge \lambda_3 - \lambda_1+1, \quad v(g_{3,2}) \ge \lambda_2 - \lambda_3,
\ees
which is modified to $v(g_{1,2}) \ge n$.  We have

\begin{align*}
2\lambda_2 - 2\lambda_1 + 2 \le 2n - 2i & = \tau \quad \text{when} \quad j = 3i, \\
n + \lambda_2 - \lambda_1 + 1 \le 3n-2i-2 & = \tau \quad \text{when} \quad j = 3i+2.
\end{align*}

\noindent
{\bf Case 6:} $\displaystyle B_\lambda = \left( \begin{array}{ccc}  * & & \\ * & * & \\ *&*& * \end{array} \right)$
\medskip

\noindent
The analog of (\ref{entryconds1}) is

\bes
v(g_{1,2}) \ge \lambda_2 - \lambda_1+1, \quad v(g_{1,3}) \ge \lambda_3 - \lambda_1+1, \quad v(g_{2,3}) \ge \lambda_3 - \lambda_2+1,
\ees
which is modified to $v(g_{1,3}) \ge n$.  We have

\begin{align*}
2\lambda_3 - 2\lambda_1 +3 \le 2n - 2i -1 & = \tau \quad \text{when} \quad j = 3i+2, \\
n + \lambda_3 - \lambda_1 + 2 \le 3n-2i & = \tau \quad \text{when} \quad j = 3i.
\end{align*}

\end{proof}

There are at most $9n^2$ choices of $\lambda \in \Sigma$ satisfying the bounds of Lemma \ref{ramlambdabound}.  Indeed, if $j = 3i+2$ then the Lemma gives $n-1 \ge \lambda_1 - \lambda_2, \lambda_2 - \lambda_3 \ge -2n+3$, and these two values determine $\lambda \in \Sigma$ uniquely.  If $j = 3i$, we have $i \ge 1$ so the Lemma likewise gives $2n-3 \ge \lambda_1 - \lambda_2, \lambda_2 - \lambda_3 \ge -n+2$.  Moreover, the bound of Lemma \ref{ramdoublecoset} may be written as $q^{3n-j+i}(1+1/q)^3$ in either case $j = 3i$ or $j = 3i+2$.  We therefore have at most $9n^2 q^{3n-j+i} (1+1/q)^3$ double cosets $K'F'^\times g K_n$ that support invariant vectors, and
\[
\dim V(\psi')^{K_n} \le 9n^2 q^{3n-j+i} (1 + 1/q)^3 \dim W(\psi').
\]
If $j$ is even, $W(\psi')$ is again obtained by inducing a character from $R'^\times K'_{j/2}$ to $K'$, and we have $\dim W(\psi') = | K' : R'^\times K'_{j/2} | = (1 - 1/q)^2 q^j$.

If $j$ is odd, set $j = 2l+1$.  As before, Howe defines $W(\psi')$ to be the induction from $R'^\times K_l'$ to $K'$ of a representation of dimension $| K_l' : U_l' K_{l+1}' |^{1/2} = q$.  This gives $\dim W(\psi') \le q | K' : R'^\times K_l' | = (1 - 1/q)^2 q^{2l+1} = (1 - 1/q)^2 q^j$.

In either case, the bound $\dim W(\psi') \le q^j$ gives
\[
\dim V(\psi')^{K_n} \le 9n^2 q^{3n-j+i} (1 + 1/q)^3 \cdot q^j = 9n^2 q^{3n+i} (1 + 1/q)^3.
\]
If $j = 3i$ then the bound $j \le 3n-2$ from Lemma \ref{ramlambdabound} gives $i \le n-1$, while if $j = 3i+2$ then $j \le 3n-2$ gives $i \le n-2$.  In either case, this completes the proof of Theorem \ref{invdim}.

\bibliographystyle{plain}
\bibliography{U}

\end{document}